\documentclass[onefignum,onetabnum]{siamart190516}

\usepackage{amsmath}
\usepackage{amsfonts}
\usepackage{amssymb}
\usepackage{mathtools}
\usepackage{physics}
\usepackage{siunitx}

\usepackage{adjustbox}
\usepackage{float}
\usepackage{graphicx}
\usepackage{rotating}
\usepackage{subcaption}
\usepackage{tikz}

\usepackage{xcolor}
\usepackage[labelfont={bf}]{caption}
\usepackage{enumerate}
\usepackage{lastpage}
\usepackage{listings}
\usepackage{lipsum}

\usepackage{verbatim}
\usepackage{amsopn}

\DeclarePairedDelimiter\ceil{\lceil}{\rceil}

\DeclarePairedDelimiter\floor{\lfloor}{\rfloor}

\newcommand{\E}{\mathbb{E}}
\newcommand{\Prob}{\mathbb{P}}
\newcommand{\totvar}{|\Delta|}

\DeclareMathOperator{\argmax}{argmax}

\DeclareMathOperator{\proj}{proj}

\newcommand\xqed[1]{%
  \leavevmode\unskip\penalty9999 \hbox{}\nobreak\hfill
  \quad\hbox{#1}}
\newcommand\qedexample{\xqed{$\triangle$}}

\newcommand{\R}{\mathbb{R}}
\newcommand{\Z}{\mathbb{Z}}

\newcommand{\naturalnumbers}{\mathbb{N}}

\date{\today} 
\allowdisplaybreaks[1]

\ifpdf
  \DeclareGraphicsExtensions{.eps,.pdf,.png,.jpg}
\else
  \DeclareGraphicsExtensions{.eps}
\fi

\newsiamremark{remark}{Remark}
\newsiamremark{hypothesis}{Hypothesis}
\newsiamremark{example}{Example}
\crefname{hypothesis}{Hypothesis}{Hypotheses}
\newsiamthm{claim}{Claim}
\headers{Pragmatic DRO of simple integer recourse models}{E. R. van Beesten, W. Romeijnders, and D. P. Morton}

\title{Pragmatic distributionally robust optimization for simple integer recourse models%
\thanks{%
Submitted to arXiv \today.
\funding{This research was supported, in part, by Northwestern University's Center for Optimization \& Statistical Learning and University of Groningen's SOM Research Institute. The research of Ward Romeijnders has been supported by Grant 451-17-034 4043 from The Netherlands Organisation for Scientific Research (NWO). }
} 
} 

\author{%
E. Ruben van Beesten%
\thanks{Department of Industrial Economics and Technology Management, Norwegian University of Science and Technology (NTNU), Trondheim, Norway (\email{ruben.van.beesten@ntnu.no}).}%
\and %
Ward Romeijnders%
\thanks{Faculty of Economics and Business, University of Groningen, Groningen, the Netherlands (\email{w.romeijnders@rug.nl}).}%
\and %
David P. Morton%
\thanks{Department of Industrial Engineering and Management Science, Northwestern University, Evanston, IL, USA (\email{david.morton@northwestern.edu}).}%
}

\ifpdf
\hypersetup{
  pdftitle={Pragmatic DRO for SIR},
  pdfauthor={E. R. van Beesten, W. Romeijnders, and D. P. Morton}
}
\fi

\begin{document}

\maketitle

\begin{abstract}
Inspired by its success for their continuous counterparts, the standard approach to deal with mixed-integer recourse (MIR) models under distributional uncertainty is to use distributionally robust optimization (DRO). We argue, however, that this modeling choice is not always justified, since DRO techniques are generally computationally extremely challenging when integer decision variables are involved. That is why we propose a fundamentally different approach for MIR models under distributional uncertainty aimed at obtaining models with improved computational tractability. For the special case of simple integer recourse (SIR) models, we show that tractable models can be obtained by \emph{pragmatically} selecting the uncertainty set. Here, we consider uncertainty sets based on the Wasserstein distance and also on generalized moment conditions. We compare our approach with standard DRO and discuss potential generalizations of our ideas to more general MIR models.
An important side-result of our analysis is the derivation of performance guarantees for \textit{convex approximations} of SIR models. In contrast with the literature, these error bounds are not only valid for continuous distribution, but hold for any distribution.
\end{abstract}

\begin{keywords}
  stochastic programming; distributional uncertainty; distributionally robust optimization; simple integer recourse;  convex approximations
\end{keywords}

\begin{AMS}
  90C15
\end{AMS}

\section{Introduction}\label{sec:introduction}

Two-stage mixed-integer recourse (MIR) models are optimization models of the form
\begin{align}
    \min_{z \in X} \Big\{ c^T z + \E^{\Prob}\big[v(\xi,z)\big] \Big\}, \label{eq:def_MIR}
\end{align}
where $X := \{z \in \Z^{\bar{n}_1}_+ \times \R^{n_1 - \bar{n}_1}_+ \ | \ A z \leq b\}$ is the first-stage feasible set, $\Prob$ is the probability distribution of the random vector $\xi$, and $v$ is the second-stage value function
\begin{align}
    v(\xi,z) := \min_{y \in Y} \{ q(\xi)^T y \ | \ T(\xi) z + W y \geq h(\xi) \}, \qquad \xi \in \Xi, \ z \in \R^{n_1}, \label{eq:def_base_model_v}
\end{align}
representing the cost of the recourse actions $y$ taken from the mixed-integer feasible set $Y := \Z_+^{\bar{n}_2} \times \R_+^{n_2 - \bar{n}_2}$. MIR models combine two computational difficulties: integer restrictions on decision variables and uncertainty about model parameters \cite{kleinhaneveld1999,kucukyavuz2017introduction}. Though these combined difficulties make MIR models extremely hard to solve in general \cite{louveaux2003stochastic}, both are often crucial in order to accurately model practical decision-making problems involving, e.g., capacity expansion, supply-chain design, and electric-grid operations \cite{kucukyavuz2017introduction}.

In many situations we need to incorporate yet another complication to obtain a useful model: distributional uncertainty \cite{rahimian2019distributionally}. That is, not only the values of some model parameters are uncertain, even their \textit{distribution} $\Prob$ is unknown. For example, we may only have a sample of historical observations believed to be from the distribution, or we may only know a few of its moments (e.g., the mean and variance). The predominant approach for dealing with distributional uncertainty is \textit{distributionally robust optimization} (DRO) \cite{rahimian2019distributionally}. However, we argue that DRO may actually be an inferior modeling approach for MIR models under distributional uncertainty, despite its success for \textit{continuous} models.

DRO deals with distributional uncertainty by defining an \textit{uncertainty set} $\mathcal{U}$ of possible distributions $\Prob$ and then optimizing under the \textit{worst-case} distribution from this set. Applying this approach to the MIR model from \eqref{eq:def_MIR} yields the \textit{distributionally robust mixed-integer recourse (DRMIR)} model
\begin{align*}
    \min_{z \in X} \big\{ c^T z + \sup_{\Prob \in \mathcal{U}} \E^{\Prob}\big[v(\xi,z)\big] \big\}.
\end{align*}
A major advantage of DRO is that its ``minimax'' structure often leads to efficiently solvable models \cite{esfahani2018data}, especially in the continuous case. This computational tractability provides a pragmatic justification for using DRO to model distributional uncertainty in the context of continuous recourse models. In the mixed-integer case, however, this justification for the DRO approach is not present, because the second-stage integer restrictions typically cause the value function $v$ to be \textit{non-convex}, making DRMIR models generally hard to solve. Nevertheless, the DRO approach is the commonly accepted and adopted modeling paradigm for dealing with distributional uncertainty, even for MIR models \cite{rahimian2019distributionally}.

The literature includes a few attempts to overcome the computational difficulties of DRMIR by deriving tractable solution approaches. These approaches combine general solution methods for DRO problems with mixed-integer programming techniques. For instance, Bansal et al. \cite{bansal2018decomposition} consider a DRMIR problem with mixed-binary variables in which they restrict the uncertainty set to distributions $\Prob$ on a finite, fixed support with $N$ realizations, $\Xi^N$. As a result, $\Prob$ can be represented by a vector $p \in [0,1]^N$. The authors consider uncertainty sets that can be represented as a polytope in $[0,1]^N$, and propose a Benders-type solution algorithm. Within the same framework regarding the uncertainty set $\mathcal{U}$, Bansal and Mehrotra \cite{bansal2019solving} and Luo and Mehrotra \cite{luo2019decomposition_method} extend the approach of \cite{bansal2018decomposition} to general disjunctive programs and mixed-integer conic programs, respectively. Xie and Ahmed \cite{xie2018distributionally} consider distributionally robust simple integer recourse with a moment-based uncertainty set in which $\mathcal{U}$ consists of all distributions with a given mean and support. The authors derive worst-case distributions and a reformulation as a second-order conic program. Finally, Kim \cite{kim2020dual} considers general DRMIR problems with a data-driven type-1 Wasserstein uncertainty set on a discretized sample space. He develops a dual decomposition approach based on the method by Carøe and Schultz \cite{Caroe1999}.

These DRMIR approaches arguably suffer from two weaknesses. First, by combining solution methods for DRO problems with mixed-integer programming techniques, they inherit computational difficulties from both areas. This inevitably limits computational efficiency of these approaches. Second, the DRMIR approaches from the literature are tailored to very specific classes of uncertainty sets, often on a discrete underlying sample space, which limits their applicability. 

In this paper, we propose a novel modeling approach to MIR models under distributional uncertainty. Instead of ``naively'' applying the DRO approach to MIR models and subsequently trying to deal with the inherent computational difficulty of the resulting DRMIR models, we propose an alternative modeling approach, which we coin \textit{pragmatic DRMIR}. The fundamental distinction in our approach is to alleviate part of the computational difficulty in the \textit{modeling} phase. 

We illustrate our pragmatic approach on the special case of \textit{simple integer recourse (SIR)} models \cite{louveaux1993stochastic}. In particular, we restrict a standard DRO uncertainty set $\mathcal{U}$ in such a way that the resulting pragmatic model is \textit{convex} and hence, much easier to solve than its ``naive'' counterpart while remaining quantifiably close to the original model. The reason for considering the special case of SIR is that we can derive deep results regarding the interplay between the distribution $\Prob$, the mixed-integer value function $v$, and convexity of the resulting model. In particular, we have an explicit characterization of the set $\mathcal{C}$ of probability distributions for which the expected recourse function
\begin{align}
    Q^{\Prob}(z) := \E^{\Prob}\big[ v(\xi,z) \big], \qquad z \in \R^{n_1}, \label{eq:def_Q}
\end{align}
is convex \cite{kleinhaneveld2006simple}. Using this characterization of ``convexifying distributions'' and accounting for proximity to a specified nominal distribution or available information on generalized moments, we construct our pragmatic uncertainty set to be a subset of $\mathcal{C}$, thus guaranteeing convexity of the resulting \textit{pragmatic DRSIR} problem. 
In fact, we will derive a so-called unit interval moving average transformation $\Gamma$ that is able to transform any probability distribution $\Prob$ to a convexifying distribution $\Prob$ in $\mathcal{C}$. Interestingly, transforming the distribution $\Prob$ using $\Gamma$ turns out to be equivalent to \emph{transforming the value function $v$}. As we will argue, this result is an important stepping stone for extending our approach to more general settings beyond SIR.

Though our pragmatic DRSIR approach is in principle applicable to arbitrary uncertainty sets, we provide a detailed analysis for two particular, popular classes of uncertainty sets:
\textit{Wasserstein balls} around a reference distribution $\Prob_0$  \cite{hanasusanto2018conic,kuhn2019wasserstein,esfahani2018data} and uncertainty sets based on (generalized) moment conditions \cite{delage2010distributionally,xie2018distributionally}. We show how to construct pragmatic counterparts of these uncertainty sets and how the resulting pragmatic DRSIR models can be reformulated as tractable, convex optimization problems. Moreover, we show how these pragmatic models relate to their corresponding standard DRSIR counterparts and we sketch algorithmic approaches for solving the resulting models. 

One important side-result of our analysis is the derivation of \textit{error bounds} for convex approximations of SIR models. Such convex approximations have been proposed in the literature as a means to overcome non-convexity of SIR models \cite{vlerk2004}. The idea is to replace the non-convex value function $v(\xi,x)$ by an approximating function $\tilde{v}(\xi,x)$, which is convex in $x$, yielding a convex approximation model. To guarantee the performance of this approach, we derive upper bounds on the approximation error in terms of the Wasserstein distance between the distribution $\Prob$ and an approximating distribution $\hat{\Prob}$. While such error bounds exist in the literature  for \textit{continuous} distributions \cite{romeijnders2016general}, our error bound is the first that holds for \textit{any} distribution. 

The remainder of this paper is organized as follows. In \cref{sec:preliminaries} we introduce SIR and discuss when these models are convex. 
\Cref{sec:wasserstein_DRSIR} investigates the fundamental properties of separability, duality, and convexity in standard Wasserstein DRSIR models. \Cref{sec:pragmatic_Wass_DRSIR} introduces our pragmatic approach in the Wasserstein setting, yielding a convex problem that can be solved much more efficiently than the standard DRSIR problem. In \cref{sec:error_bounds} we derive a stability result that is used to quantify the difference between these two approaches. Moreover, from the same stability result we obtain an error bound for convex approximations of SIR problems, which is of independent interest. In \cref{sec:pragmatic_moment_DRSIR} we extend our pragmatic approach to a setting with moment-based uncertainty sets. \Cref{sec:discussion} concludes the paper and discusses avenues for future research.

\section{Preliminaries: convexity in SIR} \label{sec:preliminaries}


\subsection{Definition of SIR}

Simple integer recourse, introduced by Louveaux and van der Vlerk~\cite{louveaux1993stochastic}, is a special case of the MIR model~\eqref{eq:def_MIR}-\eqref{eq:def_base_model_v}. It can be interpreted as a newsvendor problem with integer recourse actions. In particular, it is obtained by choosing the following values for the parameters in the MIR model: $Y = 
\Z_+^{2m}$, $W = I_{2m}$, $q = (q^+, q^-) \in \R_+^{2m}$, $h(\xi) = (\xi, -\xi)$, $\xi \in \Xi$, and $T = \begin{bmatrix} \; \; \tilde{T} \\ -\tilde{T} \end{bmatrix}$ for some $m \times n_1$ matrix $\tilde{T}$. It is not hard to see that the corresponding value function $v(\xi,z)$ only depends on $z$ through $\tilde{T} z$. To simplify notation, we introduce the so-called tender variables $x := \tilde{T} z$. Then, we can write the SIR value function as
\begin{align*}
    v(\xi,x) &:= \min_{y^+, y^-} \big\{ (q^+)^T y^+ + (q^-)^T y^-  \ | \ y^+ \geq \xi - x, \ y^- \geq x - \xi, \ y^+, y^- \in \Z_+^m \big\},
\end{align*}
for every $\xi \in \Xi, x \in \R^m$. See \cref{fig:value_function_SIR} for a plot of a univariate SIR value function $v(\cdot,x)$. 

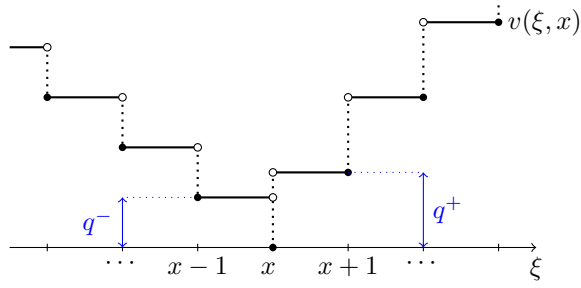
\begin{figure}
    \centering
    \begin{tikzpicture}
\draw[->] (-3.5,0) -- (3.5,0);
\draw (-3,0) node[anchor=north] {}
(-2,0) node[anchor=north] {$\cdots$}
(-1,0) node[anchor=north] {$x-1$}
(0,0) node[anchor=north] {\textcolor{white}{0}$x$ \textcolor{white}{0}}
(1,0) node[anchor=north] {$x+1$}
(2,0) node[anchor=north] {$\cdots$}
(3,0) node[anchor=north] {}
(3.5,0) node[anchor=north] {$\xi$};
\draw (-3,-.05) -- (-3,0.05)
(-2,-.05) -- (-2,0.05)
(-1,-.05) -- (-1,0.05)
(0,-.05) -- (0,0.05)
(1,-.05) -- (1,0.05)
(2,-.05) -- (2,0.05)
(3,-.05) -- (3,0.05);


\draw[thick] 
(-3.5,4*0.666) -- (-3,4*0.666)
(-3,3*0.666) -- (-2,3*0.666)
(-2,2*0.666) -- (-1,2*0.666)
(-1,1*0.666) -- (0,1*0.666)
(0,1.5*0.666) -- (1,1.5*0.666) 
(1,3*0.666) -- (2,3*0.666)
(2,4.5*0.666) -- (3,4.5*0.666);
\draw[thick, dotted]
(-3,4*0.666) -- (-3,3*0.666)
(-2,3*0.666) -- (-2,2*0.666)
(-1,2*0.666) -- (-1,1*0.666)
(0,0*0.666) -- (0,1.5*0.666)
(1,1.5*0.666) -- (1,3*0.666)
(2,3*0.666) -- (2,4.5*0.666)
(3,4.5*0.666) -- (3,5*0.666);
\node[circle,fill=black, inner sep=1pt] at (-3,3*0.666){};
\node[circle,fill=black, inner sep=1pt] at (-2,2*0.666){};
\node[circle,fill=black, inner sep=1pt] at (-1,1*0.666){};
\node[circle,fill=black, inner sep=1pt] at (0,0*0.666){};
\node[circle,fill=black, inner sep=1pt] at (1,1.5*0.666){};
\node[circle,fill=black, inner sep=1pt] at (2,3*0.666){};
\node[circle,fill=black, inner sep=1pt] at (3,4.5*0.666){};
\node[circle,fill=white, inner sep=1pt, draw] at (-3,4*0.666){};
\node[circle,fill=white, inner sep=1pt, draw] at (-2,3*0.666){};
\node[circle,fill=white, inner sep=1pt, draw] at (-1,2*0.666){};
\node[circle,fill=white, inner sep=1pt, draw] at (0,1*0.666){};
\node[circle,fill=white, inner sep=1pt, draw] at (0,1.5*0.666){};
\node[circle,fill=white, inner sep=1pt, draw] at (1,3*0.666){};
\node[circle,fill=white, inner sep=1pt, draw] at (2,4.5*0.666){};
\draw (3,4.5*0.666) node[anchor=west] {$v(\xi,x)$};

\draw[dotted,blue] (1,1.5*0.666) -- (2,1.5*0.666);
\draw[<->,blue] (2,0*0.666) -- (2,1.5*0.666);
\draw (2,0.75*0.666) node[anchor=west] {\textcolor{blue}{$q^+$}};
\draw[dotted,blue] (-1,1*0.666) -- (-2,1*0.666);
\draw[<->,blue] (-2,0*0.666) -- (-2,1*0.666);
\draw (-2,.5*0.666) node[anchor=east] {\textcolor{blue}{$q^-$}};

\end{tikzpicture}
    \vspace{-20pt}
    \caption{The SIR value function $v(\xi,x)$ as a function of $\xi$ for a given value of $x$, for the one-dimensional case ($m=1$).}
    \label{fig:value_function_SIR}
\end{figure}

One important property of SIR is that it is \textit{separable} in the $m$ dimensions. That is, $v(\xi,x) = \sum_{i=1}^m v_i(\xi_i,x_i)$, where
\begin{align} 
    v_i(\xi_i,x_i) := q_i^+ \ceil{\xi_i - x_i}^+ + q_i^- \floor{\xi_i - x_i}^-, \label{eq:SIR_value_function_rounding}
\end{align}
where $\ceil{s}^+ := \max\{ \ceil{s}, 0\}$ and $\floor{s}^- := \max\{ -\floor{s}, 0 \}$, $s \in \R$, with $\ceil{\cdot}$ and $\floor{\cdot}$ being the round-up and round-down operator, respectively. Hence,
\begin{align}
    Q^\Prob(x) := \E^{\Prob}\big[v(\xi,x)\big] = \sum_{i=1}^m  \E^{\Prob_i}\big[v_i(\xi_i, x_i)\big], \label{eq:separability}
\end{align}
where $\Prob_i = \proj_i \Prob$ is the marginal distribution of $\xi_i$, $i=1,\ldots,m$. Thus for SIR \textit{only the marginal distributions are relevant}. We use this fact pervasively in what follows.

\subsection{Convexity in SIR}

In general, the SIR expected recourse function, $Q^\Prob$, is non-convex, as a consequence of the integrality restrictions in the second-stage problem. Intuitively, this can be understood by interpreting $Q^\Prob$ as a convex combination of non-convex functions $v(\xi,\cdot)$. However, for some particular choices of $\Prob$, the SIR function is convex. From the literature we know the exact set $\mathcal{C}$ of distributions for which this is the case \cite{kleinhaneveld2006simple}. In this subsection we construct this set $\mathcal{C}$ in a novel, insightful way.

The foundation for the set $\mathcal{C}$ of ``convexifying'' distributions is a transformation $\gamma$ defined on the space of real-valued functions on $\R$. We refer to this transformation as the \textit{unit interval moving average (UIMA) transformation} and define it below, after introducing requisite notation.

\begin{definition} \label{def:measure_theory}
Consider the measurable space $(S,\mathcal{S})$, where $S = \R^m$ for some $m \in \naturalnumbers$ and $\mathcal{S}$ is the Borel sigma algebra on $S$. We define the function $e(s) := \sum_{i=1}^m (1 + |s_i|)$, $s \in S$, and the set $\mathcal{G}(S)$ as the collection of all measurable functions $g$ on $S$ with $\sup_{s \in S} \frac{|g(s)|}{e(s)} < +\infty$. Moreover, we define $\mathcal{M}(S)$, $\mathcal{M}_+(S)$, and $\mathcal{P}(S)$ as the collection of signed measures, (non-negative) measures, and probability measures $\mu$ on $(S, \mathcal{S})$, respectively, satisfying $|\int_{S} g(s) \mu(ds)| < + \infty$ for every $g \in \mathcal{G}(S)$. Finally, we define $\mathcal{F}(S)$ as the collection of all joint cumulative distribution functions (cdfs) corresponding to the probability measures in $\mathcal{P}(S)$.
\end{definition}

\begin{definition}\label{def:gamma}
We define the UIMA transformation $\gamma : \mathcal{G}(\R) \to \mathcal{G}(\R)$ by
    \begin{align*}
        \gamma(g)(s) = \int_{s - 1/2}^{s + 1/2} g(t) dt, \qquad s \in \R, \ g \in \mathcal{G}(\R),
    \end{align*}
and we use the notation $\gamma \circ g(s) := \gamma(g)(s)$, $s \in \R$, $g \in \mathcal{G}(\R)$. 
\end{definition}

As its name suggests, the UIMA transformation $\gamma$ of $g$ evaluated at $s \in \R$ can be interpreted as the (continuous) moving average of $g$ over the unit interval centered around $s$. Its relation with convexification of SIR functions is apparent from the following examples, which show that $\gamma$ can be seen as a transformation that ``repairs'' rounding, i.e., is a left inverse of the rounding function $r$ defined below.

\begin{example} \label{ex:UIMA_rounding}
    Consider the rounding function $r(s) = [s] := \max_{z \in \Z} \{ z \ | \ z - 1/2 \leq s \}$, $s \in \R$. Then, $\gamma \circ r(s) = s$, for all $s \in \R$. To see this, first suppose that s is rounded down so that $s - [s] \geq 0$. Then, on the unit interval $[s-1/2, s+1/2]$, the rounding function $r$ is two-valued and equals $r(t) = [s]$ if $s - 1/2 \leq t < [s] + 1/2$ and $r(t) = [s] + 1$ if $[s] + 1/2 \leq t  \leq s + 1/2$, and thus,
    \begin{align*}
        \gamma \circ r(s) &= \int_{s - 1/2}^{s + 1/2} [t] dt = \int_{s - 1/2}^{[s] + 1/2} [s] dt + \int_{[s] + 1/2}^{s + 1/2} ([s] + 1) dt \\
        &= \int_{s - 1/2}^{s + 1/2} [s] dt + \int_{[s] + 1/2}^{s + 1/2} dt = [s] + (s + 1/2 - ([s] + 1/2)) = s.
    \end{align*}
    By a similar derivation for the case that $s$ is rounded up, i.e., $s - [s] < 0$, we find that $\gamma \circ r (s) = s$ for all $s \in \R$. \qedexample
\end{example}

Recall from~\cref{eq:SIR_value_function_rounding} that the SIR value function $v$ can be represented in terms of rounding operations resulting from second-stage integer restrictions. Given \cref{ex:UIMA_rounding}, it is no surprise that applying the UIMA transformation to the one-dimensional function $v_i$ from the separability representation~\cref{eq:separability} of SIR, also ``repairs'' rounding here. Note that for this purpose we need to interpret $v_i$ as a function of a single real variable $s_i = \xi_i - x_i$. 

\begin{example}\label{ex:UIMA_SIR}
    Consider the function $\varphi(s) := q^+ \ceil{s}^+ + q^- \floor{s}^-$, $s \in \R$, where $q^+, q^- \in \R_+$. Then,
    \begin{align*}
        \hat{\varphi}(s) := \gamma \circ \varphi (s) = q^+ \int_{s - 1/2}^{s + 1/2} \ceil{t}^+ dt + q^- \int_{s - 1/2}^{s + 1/2} \floor{t}^- dt.
    \end{align*}
    For $s \geq -1/2$ we have 
    \begin{align*}
        \int_{s - 1/2}^{s + 1/2} \ceil{t}^+ dt &=  \int_{s - 1/2}^{\ceil{s - 1/2}} \ceil{s - 1/2} dt + \int_{\ceil{s - 1/2}}^{s + 1/2} (\ceil{s - 1/2} + 1) dt \\
        &= \int_{s - 1/2}^{s + 1/2} \ceil{s - 1/2} dt + \int_{\ceil{s - 1/2}}^{s + 1/2} dt \\
        &= \ceil{s - 1/2} + (s + 1/2 - \ceil{s - 1/2}) = s + 1/2.
    \end{align*}
    Conversely, for $s < -1/2$ we have $\int_{s - 1/2}^{s + 1/2} \ceil{t}^+ dt = 0$. Using a similar analysis for $\int_{s - 1/2}^{s + 1/2} \floor{t}^- dt$, we obtain
    \begin{align*}
        \hat{\varphi}(s) = q^+ (s + 1/2)^+ + q^- (s - 1/2)^-, \quad s \in \R.    
    \end{align*}
    Observe that the transformed function $\hat{\varphi}$ is piecewise affine and convex. \qedexample
\end{example}

\cref{ex:UIMA_SIR} shows that applying the UIMA transformation $\gamma$ to the functions $v_i$, $i=1,\ldots,m$, yields a transformed value function $\hat{v}$ that is piecewise affine and convex. In turn, this yields a convex SIR function $\hat{Q}^{\Prob}(x) := \E^{\Prob}\big[\hat{v}(\xi,x)\big]$. 

Interestingly, using the UIMA transformation $\gamma$ to instead transform the \textit{distribution} of $\xi$ also yields convex SIR functions. Consider applying $\gamma$ to a cdf $\bar{F} \in \mathcal{F}(\R)$. Let $F$ denote the resulting function, i.e.,
\begin{align*}
    F(s) := \gamma \circ \bar{F}(s) = \int_{s - 1/2}^{s + 1/2} \bar{F}(t) dt, \quad s \in \R.
\end{align*}
It is easy to verify that $F$ is also a cdf, and it has a very clear interpretation. Taking $\bar{F}$ as a starting point, $F$ is the result of smearing the probability mass at $s$ uniformly over the interval $(s - 1/2, s + 1/2)$, for every $s \in \R$. The resulting cdf $F$ can be represented as a convex combination of uniformly distributed random variables on unit intervals, as is shown in the following proposition.

\begin{proposition} \label{prop:properties_gamma}
Let $\gamma$ be the UIMA transformations from \cref{def:gamma}, and let $\xi$ be a random variable with cdf $F = \gamma(\bar{F})$, $\bar{F} \in \mathcal{F}(\R)$. Then:
\begin{enumerate}[(i)]
    \item The distribution of $\xi$ is continuous, i.e., $\xi$ has a probability density function (pdf) $f$. \label{prop:properties_gamma:continuous}
    \item We can write $f(t) = \int_{-\infty}^\infty u_s(t) d\bar{F}(s)$, $t \in \R$, where $u_s$ is the pdf of a uniformly distributed random variable on $(s-1/2,s+1/2)$, $s \in \R$. \label{prop:properties_gamma:xi_integral}
    \item The fractional value of $\xi$ (i.e., $\xi - \floor{\xi}$) is uniformly distributed on $[0,1)$. \label{prop:properties_gamma:fractional_value}    
\end{enumerate}
\end{proposition}
\begin{proof}
    Let $t \in \R$. Differentiating the cdf $F = \gamma(\bar{F})$ at $t$ yields $f(t) = \bar{F}(t + 1/2) - \bar{F}(t - 1/2)$. Moreover, 
    \begin{align*}
        \int_{-\infty}^\infty u_s(t) d\bar{F}(s) &= \int_{t - 1/2}^{t + 1/2} 1 d\bar{F}(s) = \bar{F}(t + 1/2) - \bar{F}(t - 1/2),
    \end{align*}
    which proves the first two statements. Next, \eqref{prop:properties_gamma:fractional_value} follows directly from the fact that by  \eqref{prop:properties_gamma:xi_integral}, the pdf $f$ of $\xi$ is a convex combination of pdfs of random variables distributed on unit intervals.
\end{proof}

Now we can restate the convexity result from \cite{kleinhaneveld2006simple} in terms of our UIMA transformation $\gamma$, showing that the SIR function $Q^{\Prob}$ is convex if and only if for every $i=1,\ldots,m$, the marginal cdf $F_i$ of $\xi_i$ under $\Prob$ can be represented by applying the UIMA transformation $\gamma$ to some cdf $\bar{F}_i \in \mathcal{F}(\R)$. 

\begin{lemma}\label{lemma:SIR_convexity_C}
    Consider the SIR function $Q^\Prob$ from \cref{eq:separability} 
    and let $F$ be the joint cdf of $\xi$ under $\Prob$. Suppose that $q^+ + q^- > 0$. Then, $Q^{\Prob}$ is convex if and only if $\Prob \in \mathcal{C}(\R^m)$, where
    \begin{align*}
        \mathcal{C}(\R^m) := \{ \Prob \in \mathcal{P}(\R^m) \ | \ \forall i=1,\ldots,m, \exists \bar{F}_i \in \mathcal{F}(\R) \ \text{s.t.} \ F_i = \gamma(\bar{F}_i) \}.
    \end{align*}
\end{lemma}
\begin{proof}
    The result follows directly from Theorem~1 in \cite{kleinhaneveld2006simple} and separability of SIR \cref{eq:separability}.
\end{proof}

In the rest of the paper, we will use the following notational conventions. First, we slightly abuse notation and write $\Prob = \gamma(\bar{\Prob})$ if $F = \gamma(\bar{F})$ holds for the cdfs $F, \bar{F} \in \mathcal{F}(\R)$ corresponding to the probability measures $\Prob, \bar{\Prob} \in \mathcal{P}(\R)$, respectively. Second, we simply write $\mathcal{C}$ instead of $\mathcal{C}(\R^m)$ if the dimensionality is clear from the context. 
%

Interestingly, it turns out that both approaches sketched above, i.e., using $\gamma$ to transform the value function $v$ or to transform the distribution $\bar{\Prob}$, are equivalent. 
\begin{lemma}\label{lemma:gamma_prob_func_equivalent}
    Let $g \in \mathcal{G}(\R)$ and $\bar{\Prob} \in \mathcal{P}(\R)$ be given. Then,
    \begin{align*}
        \E^{\gamma(\bar{\Prob})}\big[g(\xi)\big] = \E^{\bar{\Prob}}\big[ \gamma \circ g(\xi) \big].
    \end{align*}
\end{lemma}
\begin{proof}
    Let $f$ be the pdf corresponding to the distribution $\Prob = \gamma(\bar{\Prob})$. Then, by \cref{prop:properties_gamma} we obtain
    \begin{align*}
        \E^{\gamma(\bar{\Prob})}\big[g(\xi)\big] &= \int_{-\infty}^\infty g(s) f(s) ds = \int_{-\infty}^\infty g(s) \int_{-\infty}^{\infty} u_s(t) d\bar{F}(t) ds  \\
        &= \int_{-\infty}^\infty \int_{-\infty}^{\infty} g(s) u_s(t) ds d\bar{F}(t) = \int_{-\infty}^\infty \int_{t - 1/2}^{t + 1/2} g(s) ds d\bar{F}(t) \\ 
        &= \int_{-\infty}^\infty ( \gamma \circ g (t) ) d\bar{F}(t) = \E^{\bar{\Prob}}\big[ \gamma \circ g(\xi) \big],
    \end{align*}
    where the third equality holds by Fubini's theorem, which applies since $g \in \mathcal{G}(\R)$.
\end{proof}

\begin{corollary} \label{cor:QP=QhatPbar}
    Consider the SIR function $Q^\Prob$ and let $\Prob, \bar{\Prob} \in \mathcal{P}(\R^m)$ such that $\Prob_i = \gamma(\bar{\Prob}_i)$, $i=1,\ldots,m$. Then,
    \begin{align*}
        Q^\Prob(x) = \hat{Q}^{\bar{\Prob}}(x) := \E^{\bar{\Prob}}\big[\hat{v}(\xi,x)\big],
    \end{align*}
    where $\hat{v}(\xi,x) := \sum_{i=1}^m \hat{v}_i(\xi_i,x_i)$, $\xi,x \in \R^m$, with $\hat{v}_i(\cdot, x_i) := \gamma \circ v_i(\cdot, x_i)$, $x_i \in \R$, i.e.,
    \begin{align*}
        \hat{v}_i(\xi_i,x_i) = q^+_i (\xi_i - x_i + 1/2)^+ + q^-_i (\xi_i - x_i - 1/2 )^-, \quad \xi_i,x_i \in \R.
    \end{align*}
\end{corollary}
\begin{proof}
    The result follows by \cref{lemma:gamma_prob_func_equivalent}, \cref{ex:UIMA_SIR}, and separability of SIR \cref{eq:separability}.
\end{proof}

Both perspectives in \cref{cor:QP=QhatPbar}, i.e., transforming the distribution or transforming the value function, can be convenient, depending on the situation, and we will use both views throughout the paper.
We conclude this subsection with a generalization of the UIMA transformation $\gamma$ to its $m$-dimensional analogue, $\Gamma$, along with a corresponding analogue of \cref{cor:QP=QhatPbar}. The multivariate UIMA transformations $\Gamma$ will be useful in \cref{sec:pragmatic_Wass_DRSIR} and \cref{sec:pragmatic_moment_DRSIR} when we define our pragmatic uncertainty sets. In particular, it will allow us to transform any uncertainty set $\mathcal{U} \nsubseteq \mathcal{C}(\R^m)$ to a pragmatic uncertainty set $\hat{\mathcal{U}} := \Gamma(\mathcal{U}) \subseteq \mathcal{C}(\R^m)$.
\begin{definition}\label{def:Gamma}
    We define the $m$-dimensional UIMA transformation \linebreak $\Gamma : \mathcal{F}(\R^m) \to \mathcal{F}(\R^m)$ by
    \begin{align*}
    \Gamma \circ \bar{F}(s) := \int_{s_1 - 1/2}^{s_1 + 1/2} \cdots \int_{s_m - 1/2}^{s_m + 1/2} \bar{F}(t) dt_m \cdots dt_1, \quad s \in \R^m, \ \bar{F} \in \mathcal{F}(\R^m).
\end{align*}
\end{definition}
    Again, we slightly abuse notation and write $\Prob = \Gamma(\bar{\Prob})$ if $F = \Gamma(\bar{F})$ holds for the cdfs $F, \bar{F} \in \mathcal{F}(\R^m)$ corresponding to the probability measures $\Prob, \bar{\Prob} \in \mathcal{P}(\R^m)$, respectively.
%

\begin{remark}\label{remark:Gamma:details}
    For $m=1$, $\Gamma$ reduces to $\gamma$, so indeed $\Gamma$ generalizes $\gamma$. However, the useful relation $\gamma(\mathcal{P}(\R)) = \mathcal{C}(\R)$ does not generalize to $\Gamma$ if $m>1$. Instead, we have the strict inclusion $\Gamma(\mathcal{P}(\R^m)) \subset \mathcal{C}(\R^m)$. Fortunately, for every $\Prob \in \mathcal{C}(\R^m)$, there exists some $\tilde{\Prob} \in \Gamma(\mathcal{P}(\R^m))$ with the \textit{same marginal distributions}. Since only the marginal distributions are relevant for SIR, it follows that we can restrict ourselves to $\Gamma(\mathcal{P}(\R^m))$ instead of $\mathcal{C}(\R^m)$ without loss. See \cref{appendix:Gamma} for details. 
\end{remark}

\begin{lemma} \label{lemma:QP=QhatPbar_Gamma}
    Consider the SIR function $Q^\Prob$ and let $\Prob = \Gamma(\bar{\Prob})$ for some $\bar{\Prob} \in \mathcal{P}(\R^m)$. Then,
    \begin{align*}
        Q^\Prob(x) = \hat{Q}^{\bar{\Prob}}(x) := \E^{\bar{\Prob}}\big[\hat{v}(\xi,x)\big],
    \end{align*}
    where $\hat{v}$ is defined as in \cref{cor:QP=QhatPbar}.
\end{lemma}
\begin{proof}
    The proof follows directly from \cref{cor:QP=QhatPbar} and from \cref{lemma:Gamma_marginal} of \cref{appendix:Gamma}.
\end{proof}

\section{Standard Wasserstein DRSIR: separability, duality, and convexity} \label{sec:wasserstein_DRSIR}

In this section we define standard Wasserstein DRSIR problems. In particular, we show that standard Wasserstein DRSIR models are separable and satisfy a strong duality result, but are typically non-convex, except in special cases. Moreover, we perform a preliminary analysis that will serve as the foundation for our pragmatic Wasserstein DRSIR approach in \cref{sec:pragmatic_Wass_DRSIR} and our stability results in \cref{sec:error_bounds}. 

We define the Wasserstein DRSIR function as
\begin{align}
    \mathcal{Q}^{\Prob_0}_\varepsilon(x) := \sup_{\Prob \in \mathcal{B}_\varepsilon(\Prob_0)} \E^{\Prob}\big[v(\xi,x)\big], \label{eq:Wass_DRSIR_problem_naive}
\end{align}
where
\begin{align*}
    \mathcal{B}_\varepsilon(\Prob_0) := \{ \Prob \in \mathcal{P}(\R^m) \ | \ W_p(\Prob_0, \Prob) \leq \varepsilon \},
\end{align*}
i.e., $\mathcal{B}_\varepsilon(\Prob_0)$ is the closed ball of radius $\varepsilon$ around a given \textit{reference distribution} $\Prob_0 \in \mathcal{P}(\R^m)$, where the distance between distributions is measured using the Wasserstein distance $W_p$, defined below. The reference distribution $\Prob_0$ should be interpreted as our ``best-guess'' distribution, based on the available information. A typical choice for $\Prob_0$ is an empirical distribution corresponding to a sample of historical observations, and yields a so-called \textit{data-driven} model \cite{esfahani2018data}. 

The Wasserstein distance $W_p$ \cite{kuhn2019wasserstein} measures the distance between two probability distributions, and allows for measuring the distance between discrete and continuous distributions. This property distinguishes it from other distances between probability measures, such as $\phi$-divergences (see, e.g., \cite{bayraksan2015data}), and is important for our purposes. For example, it allows us to derive error bounds in \cref{subsec:error_bounds} that also hold for discrete distributions, unlike their analogues in the literature. 

\begin{definition} \label{def:Wasserstein_dist}
    Let $\Prob \in \mathcal{P}(S)$ and $\bar{\Prob} \in \mathcal{P}(\bar{S})$ be two probability measures on the same sample space $S = \bar{S} = \R^m$. For $p \in [1,\infty)$, we define the type-$p$ \textit{Wasserstein distance} between $\Prob$ and $\bar{\Prob}$ as
    \begin{align*}
    W_p(\Prob, \bar{\Prob}) := \left( \inf_{\pi \in \Pi(\Prob,  \bar{\Prob})} \Big\{  \int_{S \times \bar{S}} \| s - \bar{s}\|_p^p d\pi(s, \bar{s})  \Big\} \right)^{1/p},
    \end{align*}
    where $\Pi(\Prob, \bar{\Prob})$ is the collection of all joint probability distributions on $S \times \bar{S}$ with marginals $\Prob$ and $\bar{\Prob}$, and $\| \cdot \|_p$ is the $\ell_p$-norm on $\R^m$, defined as $\| z \|^p_p := \sum_{i=1}^m |z_i|^p$, $z \in \R^m$. 
\end{definition}

The Wasserstein distance can be interpreted as the cost of transporting probability mass from one distribution $\Prob$ to another $\bar{\Prob}$ using the cheapest \textit{transportation plan} $\pi$. Here, $\pi(s,\bar{s})$ represents the probability mass transported from $s$ to $\bar{s}$, and the unit cost of moving probability mass from $s$ to $\bar{s}$ is given by $\| s - \bar{s} \|_p^p$. More general definitions exist in the literature, in which $\| s - \bar{s} \|_p$ can be any metric $d(s, \bar{s})$. We choose to restrict the definition to the $\ell_p$-norm, however, since this allows for the separability result in \cref{lemma:Wass_dist_additive} below.

\subsection{Separability}

In \cref{eq:separability} we have seen that SIR is separable in the $m$ dimensions. This means that only the \textit{marginal} distributions are relevant for SIR. In this section, we use this fact to derive a separability result for the Wasserstein DRSIR function $\mathcal{Q}^{\Prob_0}_\varepsilon$. To this end, we first prove a separability result for the Wasserstein distance $W_p$, showing that in the context of Wasserstein DRSIR, the expression $W_p^p(\Prob_0,\Prob)$ (i.e., the Wasserstein distance raised to the $p$th power) is separable in its $m$ dimensions. We start with the observation that separability does \textit{not} hold in general, i.e., for arbitrary $\Prob, \, \bar{\Prob} \in \mathcal{P}(\R^m)$, we typically have that $W_p^p(\Prob, \bar{\Prob}) \neq \sum_{i=1}^m W_p^p(\Prob_i, \bar{\Prob}_i)$. We illustrate this fact with the following example.

\begin{example} \label{ex:Wasserstein_separability}
Let $p=1$ and consider $\Prob, \bar{\Prob} \in \mathcal{P}(\R^m)$, with $\Prob$ uniform on $\{(0,1), \linebreak (1,0)\}$, $\bar{\Prob}$ uniform on $\{(0,0), (0,1), (1,0), (1,1) \}$, and $m=2$. It is not hard to compute that $W_1^1(\Prob,\bar{\Prob}) = 1/2$. However, note that the marginal distributions corresponding to $\Prob$ and $\bar{\Prob}$ are equal. In particular, $\Prob_1$, $\Prob_2$, $\bar{\Prob}_1$, and $\bar{\Prob}_2$ are all uniform on $\{0,1\}$. Hence, $W_1^1(\Prob_i, \bar{\Prob}_i) = 0$, $i=1,2$, and so $\sum_{i=1}^m W_1^1(\Prob_i, \bar{\Prob}_i) = 0 \neq W_1^1(\Prob, \bar{\Prob})$.  \qedexample
\end{example}

Nevertheless, we can always find a distribution $\tilde{\Prob}$ with the same marginals as $\bar{\Prob}$ for which separability does hold.

\begin{lemma}\label{lemma:Wass_dist_additive}
    Let $\Prob \in \mathcal{P}(\R^m)$, $\bar{\Prob} \in \mathcal{P}(\R^m)$, and $p \in [1,\infty)$ be given. Then,
    \begin{align}
        W_p^p(\Prob, \bar{\Prob}) \geq \sum_{i=1}^m W_p^p(\Prob_i, \bar{\Prob}_i). \label{eq:Wasserstein_inequality}
    \end{align}
    Moreover, there exists $\tilde{\Prob} \in \mathcal{P}(\R^m)$ with marginals $\tilde{\Prob}_i = \bar{\Prob}_i$, $i=1,\ldots,m$, such that
    \begin{align}
        W_p^p(\Prob, \tilde{\Prob}) = \sum_{i=1}^m W_p^p(\Prob_i, \tilde{\Prob}_i). \label{eq:Wasserstein_equality}
    \end{align}
\end{lemma}
\begin{proof} 
    To prove inequality \cref{eq:Wasserstein_inequality}, first suppose there exists an optimal transportation plan $\pi^* \in \Pi(\Prob, \bar{\Prob})$ for  the type-$p$ Wasserstein distance $W_p(\Prob, \bar{\Prob})$. Then, it is clear that, $\pi^*_i := \proj_{S_i \times \bar{S}_i} \pi^*$ is a feasible transportation plan for $W_p(\Prob_i, \bar{\Prob}_i)$. Hence, 
    \begin{align*}
        W_p^p(\Prob,\bar{\Prob}) &= \int_{S \times \bar{S}} \| s - \bar{s} \|_p^p \, \pi^*(ds, d\bar{s}) = \int_{S \times \bar{S}} \sum_{i=1}^m  | s_i - \bar{s}_i |^p \, \pi^*(ds, d\bar{s}) \\
        &= \sum_{i=1}^m \int_{S_i \times \bar{S}_i} | s_i - \bar{s}_i |^p \, \pi^*_i(ds_i, d\bar{s}_i) \geq \sum_{i=1}^m W_p^p(\Prob_i, \bar{\Prob}_i).
    \end{align*}
    If there does not exist an optimal transportation plan $\pi^* \in \Pi(\Prob, \bar{\Prob})$, we can form a similar argument for a sequence of transportation plans with an objective converging to the optimum.
    
    Next, to prove \cref{eq:Wasserstein_equality}, assume without loss of generality that there exists an optimal transportation plan $\bar{\pi}_i$ corresponding to $W_p(\Prob_i, \bar{\Prob}_i)$, $i=1,\ldots,m$ (if these do not exist, we can again use sequences of transportation plans converging to the optimum instead). Construct a transportation plan $\tilde{\pi}$ by starting out with $\Prob$, and iteratively moving probability mass along the $i^{th}$ direction according to $\bar{\pi}_i$, $i=1,\ldots,m$. Let $\tilde{\Prob} := \proj_{\bar{S}} \tilde{\pi}$. Then, by construction, $\tilde{\Prob}_i = \bar{\Prob}_i$, $i=1,\ldots,m$. Moreover, since $\tilde{\pi}$ is a feasible transportation plan for $W_p^p(\Prob, \tilde{\Prob})$, we have
    \begin{align*}
        W_p^p(\Prob, \tilde{\Prob}) &\leq \int_{S \times \bar{S}} \| s - \bar{s} \|_p^p \, \tilde{\pi}(ds, d\bar{s}) = \sum_{i=1}^m \int_{S_i \times \bar{S}_i} | s_i - \bar{s}_i |^p \, \tilde{\pi}_i(ds_i, d\bar{s}_i)  \\
        &= \sum_{i=1}^m \int_{S_i \times \bar{S}_i} | s_i - \bar{s}_i |^p \, \bar{\pi}_i(ds_i, d\bar{s}_i) = \sum_{i=1}^m W_p^p(\Prob_i, \bar{\Prob}_i) = \sum_{i=1}^m W_p^p(\Prob_i, \tilde{\Prob}_i).
    \end{align*}
    Combined with \cref{eq:Wasserstein_inequality} this completes the proof.
\end{proof}

Together with separability of SIR from \cref{eq:separability}, the separability result in \cref{lemma:Wass_dist_additive} allows for separability of the optimization problem defining the Wasserstein DSRIR function $\mathcal{Q}_\varepsilon^{\Prob_0}(x)$ into $m$ one-dimensional problems.

\begin{proposition}\label{prop:separability_Q}
    Consider the DRSIR function $\mathcal{Q}_\varepsilon^{\Prob_0}(x)$ for $x \in \R^m$. Then, we have the separability result
    \begin{align*}
        \mathcal{Q}_{\varepsilon}^{\Prob_0}(x) &= \sup_{\varepsilon_1, \ldots, \varepsilon_m \geq 0} \bigg\{ \sum_{i=1}^m \sup_{\Prob_i \in \mathcal{P}(\R)} \big\{  \E^{\Prob_i} \big[ v_i(\xi_i,x_i)\big] \ \left | \  W_p(\Prob_{0,i}, \Prob_i) \leq \varepsilon_i \big\} \right. \ \left | \ \sum_{i=1}^m \varepsilon_i^p = \varepsilon^p \bigg\} \right. \hspace{-4pt}. 
    \end{align*}
\end{proposition}
\begin{proof}
    By definition, $\mathcal{Q}_{\varepsilon}^{\Prob_0}(x) = \sup_{\Prob \in \mathcal{P}(\R^m)} \Big\{ \sum_{i=1}^m \E^{\Prob_i}\big[ v_i(\xi_i,x_i) \big] \ | \ W_p^p(\Prob_0, \Prob) \leq \varepsilon^p \Big\}$. Assume that an optimal solution $\Prob^*$ exists for the right-hand side above (if not, an analogous proof exists in terms of an optimal sequence of solutions). Observe that $\Prob^*$ should deplete the entire Wasserstein budget, i.e., $W_p^p(\Prob_0, \Prob^*) = \varepsilon^p$, since otherwise another feasible solution $\Prob$ exists with a higher objective value. Moreover, separability should hold for this Wasserstein distance, since otherwise, by \cref{lemma:Wass_dist_additive}, there exists $\tilde{\Prob}$ with the same marginal distributions and hence, the same objective function value, that does not deplete the Wasserstein budget. It follows that we can write
    \begin{align*}
        \mathcal{Q}_{\varepsilon}^{\Prob_0}(x) &\hspace{-2pt} = \sup_{\Prob \in \mathcal{P}(\R^m)} \Big\{ \sum_{i=1}^m \E^{\Prob_i}\big[ v_i(\xi_i,x_i) \big] \ | \ \sum_{i=1}^m W_p^p(\Prob_{0,i}, \Prob_i) \leq \varepsilon^p \Big\} \\
        &\hspace{-2pt} = \sup_{\varepsilon_1, \ldots, \varepsilon_m \geq 0} \bigg\{ \sum_{i=1}^m \sup_{\Prob_i \in \mathcal{P}(\R)} \big\{  \E^{\Prob_i} \big[ v_i(\xi_i,x_i)\big] \ \left | \  W_p^p(\Prob_{0,i}, \Prob_i) \leq \varepsilon_i^p \big\} \right. \ \left | \ \sum_{i=1}^m \varepsilon_i^p = \varepsilon^p \bigg\} \right. \hspace{-4pt},
    \end{align*}
    which completes the proof.
\end{proof}

\Cref{prop:separability_Q} separates an $m$-dimensional Wasserstein DRSIR problem into $m$ one-dimensional Wasserstein DRSIR problems. Hence, the properties of general, $m$-dimensional Wasserstein DRSIR problems can be derived by first studying one-dimensional problems and then generalizing using \cref{prop:separability_Q}. We will use this strategy extensively throughout the remainder of the paper.

\subsection{Primal-dual formulation} \label{subsec:primal_formulation}

In this subsection we derive an infinite- \linebreak dimensional LP-representation of the Wasserstein DRSIR problem \cref{eq:Wass_DRSIR_problem_naive}. Moreover, we derive a dual to this primal problem and prove that strong duality holds. 

First, we show that we can replace the value function $v(\cdot,x)$ by its upper semi-continuous envelope $\bar{v}(\cdot,x)$. This is crucial for proving strong duality in \cref{lemma:strong_duality}.

\begin{lemma} \label{lemma:v_usc}
	Consider the DRSIR function $\mathcal{Q}_{\varepsilon}^{\Prob_0}$ from \cref{eq:Wass_DRSIR_problem_naive}. For every $x \in \R^m$, let $\bar{v}(\cdot,x)$ denote the upper semi-continuous envelope of $v(\cdot, x)$. That is, $\bar{v}(\bar{s},x) := \limsup_{s \to \bar{s}} v(s,x)$. Then, for every $\varepsilon > 0$ we have
	\begin{align}
	    \mathcal{Q}_{\varepsilon}^{\Prob_0}(x) &= \sup_{\Prob \in \mathcal{B}_{\varepsilon}(\Prob_0)} \E^{\Prob}\big[ \bar{v}(\xi,x) \big], \quad x \in \R. \label{eq:Q_v_bar}
	\end{align}
\end{lemma}
\begin{proof}
    Let $x \in \R^m$ be given. Since $\bar{v}(s,x) \geq v(s,x)$, $s,x \in \R^m$, it follows that $\mathcal{Q}_{\varepsilon}^{\Prob_0}(x) \leq \sup_{\Prob \in \mathcal{B}_{\varepsilon}(\Prob_0)} \E^{\Prob}\big[ \bar{v}(\xi,x) \big]$. It remains to show that the reverse inequality also holds.
    
    Let $\{ \bar{\Prob}_n \}_{n=1}^\infty$ be a maximizing sequence of solutions for the right-hand side of \cref{eq:Q_v_bar}. Let $n=1,2,\ldots$ be given and define $\tilde{\Prob}_n$ as follows. Starting out with $\bar{\Prob}_n$, for every $s \in x + \Z^m$ that has probability mass under $\bar{\Prob}_n$, move the mass to $\tilde{s}$, defined by $\tilde{s}_i := \text{sign}(s_i - x_i)\delta_n$, where $\delta_n := \big(n \Prob_n(x + \Z^m)\big)^{-1} \leq 1$. Then, by construction, $\E^{\tilde{\Prob}_n}\big[v(\xi,x)\big] = \E^{\bar{\Prob}_n}\big[\bar{v}(\xi,x)\big]$ and $W_p^p(\bar{\Prob}_n, \tilde{\Prob}_n) \leq m \bar{\Prob}_n(x + \Z^m) \delta_n = n^{-1}$. Define the distribution $\Prob^*_n := \lambda_n \tilde{\Prob}_n + (1 - \lambda_n) \Prob_0$, where $\lambda_n := \frac{\varepsilon}{\varepsilon + n^{-1/p}} \in (0,1]$. Then, using the triangle inequality for the Wasserstein distance \cite{clement2008elementary}, we have
    \begin{align*}
        W_p(\Prob_0, \Prob_n^*) &= \lambda_n W_p(\Prob_0, \tilde{\Prob}_n) \leq \lambda_n \big( W_p(\Prob_0, \bar{\Prob}_n) + W_p(\bar{\Prob}_n, \tilde{\Prob}_n) \big) \leq \lambda_n (\varepsilon + n^{-1/p}) = \varepsilon.
    \end{align*}
    Hence, $\Prob^*_n$ is feasible in the maximization problem defining $\mathcal{Q}^{\Prob_0}_\varepsilon(x)$. Moreover, we have
    \begin{align*}
        \lim_{n \to \infty} \E^{\Prob^*_n} \big[ v(\xi,x) \big] &= \lim_{n \to \infty} \Big( \lambda_n \E^{\tilde{\Prob}_n} \big[ v(\xi,x) \big] + (1 - \lambda) \E^{\Prob_0} \big[ v(\xi,x) \big] \Big) \\
        &= \lim_{n \to \infty} \lambda_n \cdot \lim_{n \to \infty} \E^{\tilde{\Prob}_n} \big[ v(\xi,x) \big] +  \lim_{n \to \infty} (1 - \lambda_n) \E^{\Prob_0} \big[ v(\xi,x) \big] \\
        &= 1 \cdot \lim_{n \to \infty} \E^{\bar{\Prob}_n} \big[ \bar{v}(\xi,x) \big] + 0 =  \sup_{\Prob \in \mathcal{B}_{\varepsilon}(\Prob_0)} \E^{\Prob}\big[ \bar{v}(\xi,x) \big].
    \end{align*}
    Hence, $\mathcal{Q}_{\varepsilon}^{\Prob_0}(x) \geq \sup_{\Prob \in \mathcal{B}_{\varepsilon}(\Prob_0)} \E^{\Prob}\big[ \bar{v}(\xi,x) \big]$, which completes the proof.
\end{proof}

Next, we provide a primal-dual infinite-dimensional LP representation of the Wasserstein DRSIR poblem \cref{eq:Wass_DRSIR_problem_naive} defining $\mathcal{Q}^{\Prob_0}(x)$. To define the primal problem, observe that instead of optimizing over the distribution $\Prob \in \mathcal{B}_\varepsilon(\Prob_0)$, we can optimize over the transportation plan $\pi \in \Pi(\Prob_0, \Prob)$ corresponding to the Wasserstein distance $W_p(\Prob_0, \Prob)$. This allows us to express the Wasserstein constraint $W_p(\Prob_0, \Prob) \leq \varepsilon$ as a \textit{linear} constraint in $\pi$:
\begin{align*}
    W_p(\Prob_0, \Prob) \leq \varepsilon &\iff \inf_{\pi \in \Pi(\Prob_0, \Prob)} \int_{S \times \bar{S}} \|s - \bar{s}\|_p^p \pi(ds, d\bar{s}) \leq \varepsilon^p.
\end{align*}
Observing that $\pi \in \Pi(\Prob_0, \Prob)$ is equivalent to the conditions $\pi \in \mathcal{M}_+(S \times \bar{S})$, $\proj_{S} \pi = \Prob_0$, and $\proj_{\bar{S}} \pi = \Prob$, we obtain following \textit{primal} infinite-dimensional LP-formulation:
\begin{align}
    \mathcal{Q}_{\varepsilon}^{\Prob_0}(x) &= \sup_{\pi \in \mathcal{M}_+(S \times \bar{S})}  \left \{ \left. \int_{S \times \bar{S}} \bar{v}(\bar{s}, x) \pi(ds, d\bar{s}) \  \right | \ \proj_{S} \pi = \Prob_0, \right. \label{eq:Wass_DRSIR_problem_primal}  \\ 
    &\qquad \qquad \qquad \qquad \qquad \qquad \qquad \qquad \left.  \ \int_{S \times \bar{S}} \|s - \bar{s}\|_p^p \pi(ds, d\bar{s}) \leq \varepsilon^p \right \}. \nonumber 
\end{align}
Introducing the dual variables $\nu \in \mathcal{G}(S)$ and $\lambda \in \R_+$ for the first and second constraint, respectively, we can write the corresponding \textit{dual} problem as
\begin{align}
    \inf_{\substack{\lambda^p \in \R_+, \\ \ \nu \in \mathcal{G}(S)}} \Big\{ \lambda \varepsilon + \int_{S} \nu(s) \Prob_0(ds) \ \left | \ \nu(s) + \lambda \|s - \bar{s}\|_p^p \geq \bar{v}(\bar{s}, x), \ (s,\bar{s}) \in S \times \bar{S} \Big\}. \right. \label{eq:Wass_DRSIR_problem_dual}
\end{align}
We can show that strong duality holds between these primal-dual formulations.

\begin{lemma}\label{lemma:strong_duality}
Let $x \in \R^m$ and $\varepsilon > 0$ be given. Then, strong duality holds for the primal-dual pair \cref{eq:Wass_DRSIR_problem_primal}--\cref{eq:Wass_DRSIR_problem_dual}. Consequently, we can write
\begin{align}
    \mathcal{Q}_{\varepsilon}^{\Prob_0}(x) &= \inf_{\lambda \in \R_+} \Big\{ \lambda \varepsilon^p + \int_{S} \nu^\lambda(s,x) \Prob_0(ds) \Big\}, \label{eq:Wass_DRSIR_strong_duality}
\end{align}
where
\begin{align}
    \nu^\lambda(s,x) := \sup_{\bar{s} \in \bar{S}} \{ \bar{v}(\bar{s},x) - \lambda \| s - \bar{s} \|_p^p \}, \quad s \in S. \label{eq:nu_sup}
\end{align}
\end{lemma}
\begin{proof}
    Using the fact that $\bar{v}(\cdot,x)$ is upper semi-continuous, strong duality (i.e., equality) between the primal-dual pair \cref{eq:Wass_DRSIR_problem_primal}--\cref{eq:Wass_DRSIR_problem_dual} follows from Theorem~1 in \cite{gao2016distributionally}. Observing that for every $s \in S$, $\nu(s)$ enters the dual objective with a non-negative coefficient, it follows that $\nu(s) = \nu^\lambda(s,x)$ is optimal for every $\lambda \in \R_+$ and hence, the right-hand side of \cref{eq:Wass_DRSIR_strong_duality} is equal to the dual in \cref{eq:Wass_DRSIR_problem_dual}. This completes the proof.
\end{proof}

Note that from an optimization perspective, the non-concavity and non-smooth- \linebreak ness of $\bar{v}(\cdot,x)$ in the right-hand side of \cref{eq:nu_sup} complicate the dual formulation of $\mathcal{Q}^{\Prob_0}_\varepsilon(x)$. In \cref{sec:pragmatic_Wass_DRSIR,sec:pragmatic_moment_DRSIR} we will show how our pragmatic approach overcomes this issue.

\subsection{Convexity of $\mathcal{Q}_\varepsilon^{\Prob_0}$} \label{subsec:convexity_result}

We conclude this section by investigating conditions under which the Wasserstein DRSIR function $\mathcal{Q}_\varepsilon^{\Prob_0}$ is convex. To this end, we investigate the special case $p=1$, making use of the primal-dual formulation derived above. Initially, we anticipate that $\mathcal{Q}_\varepsilon^{\Prob_0}$ will be non-convex as a consequence of non-convexity of the underlying value function $v(\xi,x)$, as a function of $x$. Indeed, we show that $\mathcal{Q}_\varepsilon^{\Prob_0}$ is typically non-convex, except for values of $\varepsilon$ that are sufficiently large under some special conditions on the cost vector $q$.

From \cref{eq:nu_sup} and the fact that $\nu$ has a non-negative coefficient in the dual objective it follows that $\lambda \geq \| q \|_\infty$; otherwise, we obtain $\nu^\lambda(s,x) = +\infty$ and hence, a dual objective of $+\infty$. As a result, given $\lambda \geq \| q \|_\infty$ we can compute an explicit expression for $\nu^\lambda(s,x)$:
\begin{align}
    \nu^\lambda(s,x) = \bar{v}(s,x) + r^\lambda(s,x), \quad s,x \in \R^m, \label{eq:def_nu}
\end{align}
with $r^\lambda(s,x) := \sum_{i=1}^m r^\lambda(s_i,x_i)$, $s,x \in \R^m$, where, with $x_i^\lambda := x_i - (q^+_i - q^-_i)/\lambda$,
\begin{align}
    r^\lambda_i(s_i,x_i) &:= \begin{cases}
    \max\{ q^-_i - \lambda (s_i - \floor{s_i}_{x_i}), 0 \} &\text{if } s_i \leq x_i^\lambda, \\
    \max\{ q^+_i - q^-_i - \lambda (\ceil{s_i}_{x_i} - s_i), 0\} &\text{if } x_i^\lambda <  s_i < x_i, \\
    \max\{ q^+_i - \lambda (\ceil{s_i}_{x_i} - s_i), 0\} &\text{if } s_i \geq x_i,
    \end{cases} \label{eq:def_r_plus}
\end{align}
if $q_i^+ \geq q_i^-$, and 
\begin{align}
    r^\lambda_i(s_i,x_i) &:= \begin{cases}
    \max\{ q^-_i - \lambda (s_i - \floor{s_i}_{x_i}), 0 \} &\text{if } s_i \leq x_i, \\
    \max\{ q^-_i - q^+_i - \lambda (s_i - \floor{s_i}_{x_i}), 0\} &\text{if } x_i <  s_i < x_i^\lambda, \\
    \max\{ q^+_i - \lambda (\ceil{s_i}_{x_i} - s_i), 0\} &\text{if } s_i \geq x_i^\lambda,
    \end{cases} \label{eq:def_r_minus}
\end{align}
if $q_i^+ \leq q_i^-$. See \cref{fig:nu_star} for an illustration. Using this representation of $\nu^\lambda$, we can prove that $\lambda = \| q \|_\infty$ is optimal.

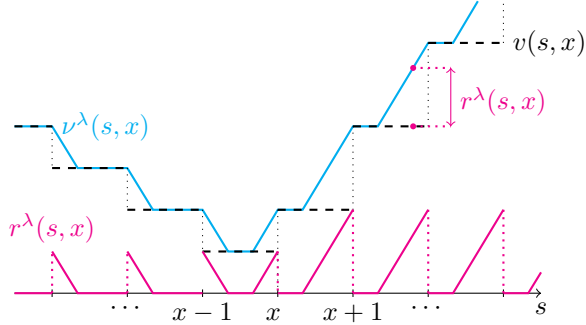
\begin{figure}
    \centering
    \begin{tikzpicture}
\draw[->] (-3.5,0) -- (3.5,0);
\draw (-3,0) node[anchor=north] {}
(-2,0) node[anchor=north] {$\cdots$}
(-1,0) node[anchor=north] {$x-1$}
(0,0) node[anchor=north] {\textcolor{white}{0}$x$ \textcolor{white}{0}}
(1,0) node[anchor=north] {$x+1$}
(2,0) node[anchor=north] {$\cdots$}
(3,0) node[anchor=north] {}
(3.5,0) node[anchor=north] {$s$};
\draw (-3,-.05) -- (-3,0.05)
(-2,-.05) -- (-2,0.05)
(-1,-.05) -- (-1,0.05)
(0,-.05) -- (0,0.05)
(1,-.05) -- (1,0.05)
(2,-.05) -- (2,0.05)
(3,-.05) -- (3,0.05);


\draw[thick, cyan] 
(-3.5,4*0.555) -- (-3,4*0.555) -- (-2.66,3*0.555) -- (-2,3*0.555) -- (-1.66,2*0.555) -- (-1,2*0.555) -- (-.66,1*0.555) -- (-.33,1*0.555) -- (0,2*0.555) -- (0.33,2*0.555) -- (1,4*0.555) -- (1.33,4*0.555) -- (2,6*0.555) -- (2.33,6*0.555) -- (2.66,7*0.555);
\draw (-3,4*0.555) node[anchor=west] {\textcolor{cyan}{$\nu^\lambda(s,x)$}};

\draw[thick, dashed] 
(-3.5,4*0.555) -- (-3,4*0.555)
(-3,3*0.555) -- (-2,3*0.555)
(-2,2*0.555) -- (-1,2*0.555)
(-1,1*0.555) -- (0,1*0.555)
(0,2*0.555) -- (1,2*0.555) 
(1,4*0.555) -- (2,4*0.555)
(2,6*0.555) -- (3,6*0.555);
\draw[dotted]
(-3,4*0.555) -- (-3,3*0.555)
(-2,3*0.555) -- (-2,2*0.555)
(-1,2*0.555) -- (-1,1*0.555)
(0,0*0.555) -- (0,2*0.555)
(1,2*0.555) -- (1,4*0.555)
(2,4*0.555) -- (2,6*0.555)
(3,6*0.555) -- (3,7*0.555);
\draw (3,6*0.555) node[anchor=west] {$v(s,x)$};

\node at (1.8,5.4*0.555) [circle,fill,magenta, inner sep = 0.8 pt]{};
\node at (1.8,4.0*0.555) [circle,fill,magenta, inner sep = 0.8 pt]{};
\draw[dotted,thick,magenta] (1.8,5.4*0.555) -- (2.3,5.4*0.555);
\draw[dotted,thick,magenta] (1.8,4.0*0.555) -- (2.3,4.0*0.555);
\draw[<->,magenta] (2.3,5.4*0.555) -- (2.3,4.0*0.555);
\draw (2.3,4.6*0.555) node[anchor=west] {\textcolor{magenta}{$r^\lambda(s,x)$}};

\draw[thick, magenta] (-3.5,0*0.555) -- (-3,0*0.555);
\draw[thick, magenta, dotted] (-3,0*0.555) -- (-3,1*0.555);
\draw[thick, magenta] (-3,1*0.555) -- (-2.66,0*0.555) -- (-2,0*0.555);
\draw[thick, magenta, dotted] (-2,0*0.555) -- (-2,1*0.555);
\draw[thick, magenta] (-2,1*0.555) -- (-1.66,0*0.555) -- (-1,0*0.555);
blue, dotted] (-1,0*0.555) -- (-1,1*0.555);
\draw[thick, magenta] (-1,1*0.555) -- (-0.66,0*0.555) -- (-.33,0*0.555) -- (0,1*0.555);
\draw[thick, magenta, dotted] (0,1*0.555) -- (0,0*0.555);
\draw[thick, magenta] (0,0*0.555) -- (.33,0*0.555) -- (1,2*0.555);
\draw[thick, magenta, dotted] (1,2*0.555) -- (1,0*0.555);
\draw[thick, magenta] (1,0*0.555) -- (1.33,0*0.555) -- (2,2*0.555);
\draw[thick, magenta, dotted] (2,2*0.555) -- (2,0*0.555);
\draw[thick, magenta] (2,0*0.555) -- (2.33,0*0.555) -- (3,2*0.555);
\draw[thick, magenta, dotted] (3,2*0.555) -- (3,0*0.555);
\draw[thick, magenta] (3,0*0.555) -- (3.33,0*0.555) -- (3.5,0.5*0.555);
\draw (-3,1*0.555) node[anchor=south] {\textcolor{magenta}{$r^\lambda(s,x)$}};

\end{tikzpicture}
    \vspace{-20pt}
    \caption{A plot of $\nu^\lambda(\cdot,x)$, $v(\cdot,x)$, and their difference $r^\lambda(\cdot,x)$ for the one-dimensional case ($m=1$) with $q^+=2$, $q^- = 1$, and $\lambda = 3$.}
    \label{fig:nu_star}
\end{figure}

\begin{lemma} \label{lemma:DRSIR_function_dual_expression}
Consider the DRSIR function $\mathcal{Q}^{\Prob_0}_\varepsilon$. If $\varepsilon^p \geq \sum_{i=1}^m \| q_i \|_\infty$, where $q_i=(q_i^+,q_i^-)$ then
\begin{align*}
    \mathcal{Q}_\varepsilon^{\Prob_0}(x) = \E^{\Prob_0}\big[\nu^{\|q\|_\infty}(\xi,x)\big] + \|q\|_\infty \cdot \varepsilon,
\end{align*}
where $\nu^{\|q\|_\infty}$ is defined as in \cref{eq:def_nu}.
\end{lemma}
\begin{proof}
    It suffices to show that $\lambda = \| q \|_\infty$ is optimal in \eqref{eq:Wass_DRSIR_strong_duality}. Suppose that $\lambda = \| q \|_\infty + \delta$ with $\delta > 0$. We will compare the corresponding objective value with the objective value corresponding to $\lambda = \| q \|_\infty$. For the first term in the objective we have $\lambda \varepsilon^p - \| q \|_\infty \varepsilon^p = \delta \varepsilon^p \geq \delta \frac{ \sum_{i=1}^m \|q_i\|_\infty}{\| q \|_\infty}$, where the last inequality follows from the assumption on $\varepsilon^p$. 
    
    Next, we compare the second term in the objective. We have 
    \begin{align*}
        \int_{S} \nu^{\| q \|_\infty}(s,x) \Prob_0(ds) - \int_{S} \nu^\lambda(s,x) \Prob_0(ds) &= \int_{S} \big( \nu^{\| q \|_\infty}(s,x) - \nu^\lambda(s,x) \big) \Prob_0(ds) \\
        &= \int_{S} \big( r^{\| q \|_\infty}(s,x) - r^\lambda(s,x) \big) \Prob_0(ds).
    \end{align*}
    We have $r^{\| q \|_\infty}(s,x) - r^\lambda(s,x) = \sum_{i=1}^m \big( r_i^{\| q \|_\infty}(s_i,x_i) - r_i^\lambda(s_i,x_i)\big)$. Let $i=1,\ldots,m$, be given. Then, by definition of $r^\lambda_i$ (and using \cref{fig:nu_star}), it is not hard to see that
    \begin{align*}
        r_i^{\| q \|_\infty}(s_i,x_i) - r_i^\lambda(s_i,x_i)\big) &\leq \delta \|q_i\|_\infty/\|q\|_\infty. 
    \end{align*}
    Substituting this above yields
    \begin{align*}
        \int_{S} \nu^{\| q \|_\infty}(s,x) \Prob_0(ds) - \int_{S} \nu^\lambda(s,x) \Prob_0(ds) &\leq \delta \frac{ \sum_{i=1}^m \|q_i\|_\infty}{\| q \|_\infty}.
    \end{align*}
    Comparing the two terms, it follows that the objective function value for $\lambda = \| q \|_\infty + \delta$ is higher than for $\lambda = \| q \|_\infty$. Since $\delta > 0$ is chosen arbitrarily, it follows that $\lambda = \| q \|_\infty$ is optimal. This completes the proof.
\end{proof}

From \cref{lemma:DRSIR_function_dual_expression}, it follows that convexity of $\mathcal{Q}^{\Prob_0}_\varepsilon$ depends on convexity of \linebreak $\nu^{\| q \|_\infty}(\xi,\cdot)$. In \cref{fig:nu_star} we observe that in the one-dimensional case, the slope of $\nu^{\lambda}(\cdot,x)$ is either $0$, $\lambda$, or $-\lambda$, and we know that $\lambda \geq \max\{q^+, q^-\}$. From this fact, we have that $\nu^{\lambda}(\cdot,x)$ is convex if and only if $\lambda = \max\{q^+, q^-\}$ and \linebreak $q^+, q^- \in \{0, \max\{q^+, q^-\} \}$. By symmetry, $\nu^{\lambda}(\xi,\cdot)$ is convex under the exact same conditions. Generalizing this result to multiple dimensions, we obtain the following. 

\begin{theorem} \label{thm:DRSIR_convexity}
    Consider the DRSIR function $\mathcal{Q}_\varepsilon^{\Prob_0}$ from \cref{eq:Wass_DRSIR_problem_naive}. Then,  $\mathcal{Q}_\varepsilon^{\Prob_0}$ is convex if $q_i^+, q_i^- \in \{0, \|q\|_\infty \}$, $i=1,\ldots,m$, and $\varepsilon \geq m$.
\end{theorem}
\begin{proof}
    Let $i=1,\ldots,m$, be given and fix $x_i \in \R$. We first derive conditions under which $\nu^\lambda_i(\cdot,x_i)$ is convex. Using the fact that $\lambda \geq \| q \|_\infty$, we see that $\nu^\lambda_i(\cdot,x_i)$ is convex on $[x^\lambda_i, \infty)$ if $q^+_i \in \{0,\lambda\}$, and moreover, it is non-decreasing on that interval. Similarly, $\nu^\lambda_i(\cdot,x_i)$ is convex on $(-\infty, x^\lambda_i]$ if $q^-_i \in \{0,\lambda\}$, and is non-increasing on that interval. Combining these facts with the observation that $\nu^\lambda_i(\cdot,x_i)$ is continuous, we find that $\nu^\lambda_i(\cdot,x_i)$ is convex if $q^+_i, q^-_i \in \{0,\lambda\}$. By symmetry, it follows that for any fixed $s_i \in \R$,  $\nu^\lambda_i(s_i,\cdot)$ is convex if $q^+_i, q^-_i \in \{0,\lambda\}$. Since $\nu^\lambda(s,x) = \sum_{i=1}^m \nu^\lambda_i(s_i,x_i)$, it follows that $\nu^\lambda(s,\cdot)$ is convex if $q^+_i, q^-_i \in \{0,\lambda\}$ for every $i=1,\ldots,m$. Finally, substituting $\lambda = \| q \|_\infty$, using \cref{lemma:DRSIR_function_dual_expression}, and using the fact that the expected value of convex functions is convex, it follows that $\mathcal{Q}^{\Prob_0}_\varepsilon$ is convex provided $q^+_i, q^-_i \in \{0,\| q \|_\infty\}$, $i=1,\ldots,m$.
\end{proof}

So surprisingly, \cref{thm:DRSIR_convexity} shows that there are conditions under which Wasserstein DRSIR is in fact convex, despite the integer restrictions in the second stage. These conditions include the important special case of the single-product newsvendor problem, for which we can find an analytical solution, as presented in the following example.

\begin{example} \label{ex:newsvendor_problem}
    Consider the Wasserstein DRSIR problem \cref{eq:Wass_DRSIR_problem_naive} with $p=1$, $m=1$, $X = \R$, $q^+ > c > 0$, $q^- = 0$, and $\varepsilon \geq 1$. Then, by \cref{thm:DRSIR_convexity} it follows that $\mathcal{Q}^{\Prob_0}_\varepsilon(x) = \E^{\Prob_0}\big[ q^+(\xi - x + 1)^+\big] + q^+ \varepsilon$. Thus an optimal solution to \cref{eq:Wass_DRSIR_problem_naive} is given by $x^* := \min_{z \in \R}\big\{ \Prob_0\{ \xi \leq z \} \geq c/q^+\big\}$. \qedexample
\end{example}

Nevertheless, for multi-dimensional problems, the conditions in \cref{thm:DRSIR_convexity} are quite restrictive. In particular, we require that all $2m$ elements of $q$ can only have one common positive value and that $\varepsilon$ is rather large. While \cref{thm:DRSIR_convexity}'s conditions are not necessary for $\mathcal{Q}_\varepsilon^{\Prob_0}$ to be convex, we do not expect convexity to hold if the hypotheses of the theorem are not satisfied. Specifically, if $\varepsilon \geq m$, but $q^{\pm}_i \notin \{0,\| q \|_\infty\}$ for some $q$-coefficient, then by \cref{fig:nu_star}, the function $\nu^\lambda$ will be non-convex in the corresponding direction. Hence, we expect $\mathcal{Q}_\varepsilon^{\Prob_0}(x) = \E^{\Prob_0}\big[\nu^{\|q\|_\infty}(\xi,x)\big] + \|q\|_\infty \cdot \varepsilon$ to be non-convex too, even though exceptions occur for very particular choices of $\Prob_0$. We illustrate the general behavior of $\mathcal{Q}^{\Prob_0}_\varepsilon$ by two examples in \cref{fig:convexity}. As a result we infer that, generally speaking, we expect the DRSIR problem to be \textit{non-convex}.

\begin{figure}%
    \centering
    \subfloat[\centering $q^+ = 2$, $q^- = 0$]{{\includegraphics[scale=0.6921]{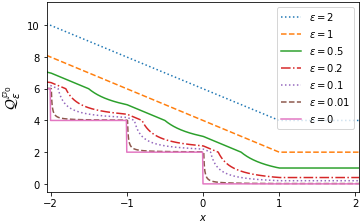} }}\label{fig:convexity:one-sided}%
    \subfloat[\centering $q^+ = 2$, $q^- = 1$]{{\includegraphics[scale=0.6921]{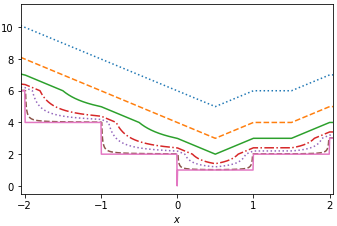} }}\label{fig:convexity:two-sided}%
    \caption{Plots of the one-dimensional Wasserstein DRSIR function $\mathcal{Q}_\varepsilon^{\Prob_0}$ under the distribution $\Prob_0\{\xi = 0\} = 1$.}%
    \label{fig:convexity}%
\end{figure}

\section{Pragmatic Wasserstein DRSIR} \label{sec:pragmatic_Wass_DRSIR}

We now introduce our pragmatic approach that aims to address the non-convexity issues of standard DRSIR outlined in the previous section. We first outline the main approach and then apply it to a
setting with a Wasserstein-based uncertainty set. In \cref{sec:pragmatic_moment_DRSIR} we show that the same approach works for (generalized) moment-based DRSIR problems as well.

\subsection{Defining the pragmatic uncertainty set} \label{subsec:defining_pragmatic_U}

The main idea in our pragmatic approach for SIR is to model the distributional uncertainty slightly differently than traditional approaches, in order to gain convexity. We start with an arbitrary uncertainty set $\mathcal{U}$ that is interesting from a modeling perspective, which could be a typical uncertainty set, e.g., a Wasserstein ball or a moment-based set. Then, we transform $\mathcal{U}$ to a structurally similar pragmatic uncertainty set $\bar{\mathcal{U}}$, which results in a \textit{convex} model. In particular, we construct $\bar{\mathcal{U}}$ so that it is a subset of $\mathcal{C}$, the set of ``convexifying'' distributions from \cref{lemma:SIR_convexity_C}, yielding a pragmatic DRSIR function that is convex.

\begin{lemma} \label{lemma:Q_U_convex}
Suppose that $\bar{\mathcal{U}} \subseteq \mathcal{C}$, where $\mathcal{C}$ is defined in \cref{lemma:SIR_convexity_C}. Then, the DRSIR function
\begin{align*}
    \mathcal{Q}_{\bar{\mathcal{U}}}(x) := \sup_{\Prob \in \bar{\mathcal{U}}} \E^{\Prob}\big[v(\xi,x)\big], \quad x \in \R^m,
\end{align*}
is convex in $x$.
\end{lemma}
\begin{proof}
    For every $\Prob \in \bar{\mathcal{U}}$, we have $\Prob \in \mathcal{C}$ and thus, by \cref{lemma:SIR_convexity_C}, $\E^{\Prob}\big[v(\xi,x)\big]$ is convex as a function of $x$. So $\mathcal{Q}_{\bar{\mathcal{U}}}$ is the pointwise supremum of convex functions, which is convex.
\end{proof}

How exactly do we transform $\mathcal{U}$ to a subset of $\mathcal{C}$? From \cref{subsec:convexity_result}, two approaches arise: intersecting $\mathcal{U}$ with $\mathcal{C}$ or transforming $\mathcal{U}$ by the UIMA transformation $\Gamma$, i.e.,
\begin{align*}
    \tilde{\mathcal{U}} := \mathcal{U} \cap \mathcal{C} \qquad \text{or} \qquad \hat{\mathcal{U}} := \Gamma(\mathcal{U}). 
\end{align*}
While in principle these approaches can apply to any initial uncertainty set, in what follows we restrict attention to two special cases: Wasserstein-based uncertainty sets (this section) and generalized moment-based uncertainty sets (\cref{sec:pragmatic_moment_DRSIR}).

We now restrict ourselves to the setting in which the initial DRSIR model has an uncertainty set that is a Wasserstein ball around an empirical distribution: $\mathcal{U} := \mathcal{B}_\varepsilon(\Prob_0)$. We start by showing that both $\tilde{\mathcal{U}}$ and $\hat{\mathcal{U}}$ are quantifiably ``close'' to the initial set $\mathcal{U}$.

\begin{lemma} \label{lemma:Wass_dist_max_m/4}
Let $\Prob \in \mathcal{P}(\R^m)$ be given, $\mathcal{C}$ be defined in \cref{lemma:SIR_convexity_C}, $\Gamma$ be specified in \cref{def:Gamma}, and $W_p$ be given by \cref{def:Wasserstein_dist}. Then, 
\begin{align*}
    \inf_{\tilde{\Prob} \in \mathcal{C}} W_p(\Prob,\tilde{\Prob}) \leq W_p(\Prob,\Gamma(\Prob)) \leq m/4.
\end{align*}
\end{lemma}
\begin{proof} 
    The first inequality holds since $\Gamma(\Prob) \in \mathcal{C}$ by definition of $
    \mathcal{C}$ if $m=1$ and by \cref{lemma:Gamma_inclusions} of \cref{appendix:Gamma} if $m > 1$. For the second inequality, let $\pi$ be defined through the continuous conditional distribution $\pi_{\bar{S} | S}$ with conditional pdf $f_{\bar{S} | S}$, defined as
    \begin{align*}
        f_{\bar{S} | S}(s | \bar{s}) = \begin{cases}
        1 &\text{if } \bar{s} \in s + (-1/2, 1/2)^m, \\
        0 &\text{otherwise.}
        \end{cases}
    \end{align*}
    Then clearly, $\pi \in \Pi(\Prob, \Gamma(\Prob))$. Hence,  writing $\pi_S := \proj_S \pi$, we have:
    \begin{align*}
        W_p(\Prob,\Gamma(\Prob)) & \hspace{-2pt} \leq \int_{S \times \bar{S}} \| s - \bar{s} \|_1 \pi(ds, d\bar{s}) = \int_{S} \int_{\bar{S}} \| s - \bar{s} \|_1 \pi_{\bar{S}|S}(d\bar{s}|s) \pi_S(ds) \\
        & \hspace{-2pt} = \int_{S} \int_{\bar{S}} \| s - \bar{s} \|_1 f_{\bar{S}|S}(\bar{s}|s) d\bar{s} \pi_S(ds) = \int_{S} \int_{s + (-1/2, 1/2)^m}  \| s - \bar{s} \|_1 d\bar{s} \, \pi_S(ds) \\
        & \hspace{-2pt} = \int_{S} \Big( \sum_{i=1}^m \int_{s_i - 1/2}^{s_i + 1/2}  | s_i - \bar{s}_i | d\bar{s}_i \Big)  \pi_S(ds) = \int_{S} m \cdot \frac{1}{4}  \pi_S(ds) = m/4.
    \end{align*}
\end{proof}

From the first inequality in \cref{lemma:Wass_dist_max_m/4}---and from knowing that $\Gamma$ smears probability mass locally within unit intervals---we learn that for any $\Prob \in \mathcal{P}(\R^m)$ there is always some $\tilde{\Prob} \in \mathcal{C}$ close to $\Prob$. Hence, all distributions that we ``cut off'' when we intersect $\mathcal{B}_\varepsilon(\Prob_0)$ with $\mathcal{C}$ can be approximated reasonably well by a distribution in $\tilde{\mathcal{U}}$. Similarly, by the second inequality in \cref{lemma:Wass_dist_max_m/4}, any distribution in $\mathcal{B}_\varepsilon(\Prob_0)$ can be approximated reasonably well by a distribution in $\hat{\mathcal{U}}$ in the sense made precise by the lemma's bound, $m/4$. So if the original Wasserstein ball is a reasonable choice for the uncertainty set from a modeling perspective, then one can argue that $\tilde{\mathcal{U}}$ and $\hat{\mathcal{U}}$ are as well, at least approximately.

Which of the two options, $\tilde{\mathcal{U}}$ or $\hat{\mathcal{U}}$, is preferable is not clear a priori. The set $\tilde{\mathcal{U}}$ has the advantage that we know that $\tilde{\mathcal{U}} \subseteq \mathcal{U}$, so this approach is less conservative than standard Wasserstein DRSIR. However, a drawback is the fact that we cannot guarantee that $\tilde{\mathcal{U}}$ is non-empty if $\varepsilon < m/4$. Conversely, while ensuring convexity, the set $\hat{\mathcal{U}}$ is not necessarily a subset of $\mathcal{U}$, so it is not clear whether this approach is less or more conservative than standard Wasserstein DRSIR. However, we can ensure that $\hat{\mathcal{U}}$ is non-empty.

In \cref{subsec:primal/dual_formulations} we argue that the set $\hat{\mathcal{U}}$ is preferable from a pragmatic point of view because it results in a simpler, tractable optimization problem. Nevertheless, we also present the analysis for $\tilde{\mathcal{U}}$ because it provides valuable insight. In particular, it highlights how explicit inclusion of the constraint $\Prob \in \mathcal{C}$ from $\tilde{\mathcal{U}}: = \mathcal{U} \cap \mathcal{C}$ complicates the resulting model. This issue may carry over to settings beyond Wasserstein DRSIR, though not always as we show in \cref{sec:pragmatic_moment_DRSIR}.

\subsection{Dual formulations and tractability} \label{subsec:primal/dual_formulations}

We now derive dual formulations for the pragmatic DRSIR functions
\begin{align}
    \tilde{\mathcal{Q}}^{\Prob_0}_\varepsilon(x) := \sup_{\Prob \in \tilde{\mathcal{U}}}  \E^{\Prob}\big[v(\xi,x)\big], \quad \text{and} \quad  \hat{\mathcal{Q}}^{\Prob_0}_\varepsilon(x) := \sup_{\Prob \in \hat{\mathcal{U}}}  \E^{\Prob}\big[v(\xi,x)\big], \quad  x \in \R^m, \label{eq:def_Q_tilde_and_hat}
\end{align}
and use these to investigate tractability of the resulting pragmatic DRSIR problems. First, for $\tilde{\mathcal{U}}$ we show that, despite convexity of the resulting $\tilde{\mathcal{Q}}^{\Prob_0}_\varepsilon$ it still suffers from tractability issues due to the explicit inclusion of the constraint $\Prob \in \mathcal{C}$. Next, we show that $\hat{\mathcal{U}}$ does yield a tractable model and we sketch a row generation algorithm that may be used to solve the resulting pragmatic DRSIR model. Moreover, we show that if $p=1$, the pragmatic DRSIR model reduces to a continuous recourse problem under a known distribution.

In what follows, we use the fact that by an analogue of \cref{prop:separability_Q}, both $\tilde{\mathcal{Q}}^{\Prob_0}_\varepsilon$ and $\hat{\mathcal{Q}}^{\Prob_0}_\varepsilon$ are separable. Hence, we restrict ourselves to the one-dimensional case ($m=1$).

\subsubsection{Option 1: $\tilde{\mathcal{U}}$} \label{subsubsec:U_tilde}

First, we consider the pragmatic uncertainty set $\tilde{\mathcal{U}} := \mathcal{B}_\varepsilon(\Prob_0) \cap \mathcal{C}$. We will show that this choice does not yield a tractable model, despite convexity of $\tilde{\mathcal{Q}}^{\Prob_0}_\varepsilon$. For this purpose, we consider the special case where $\Prob_0$ is discrete on a finite set $\Xi^N := \{ \xi^1, \ldots, \xi^N\}$ of points, with $p^k := \Prob_0\{\xi = \xi^k\}$, $k=1,\ldots,N$. We derive a primal-dual pair for this setting. To this end, we make a few observations.

First, since any $\Prob \in \mathcal{C}$ is continuous, we can express any transportation plan $\pi \in \Pi(\Prob_0, \Prob)$ in terms of (scaled) conditional pdfs, $f^k$, $k =1,\ldots,N$, where $f^k$ is 
associated with $\proj_{\bar{S}} \pi(\xi^k, \cdot)$, $k=1,\ldots,N$. Note that $f^k$ is a function from $\R$ to $\R_+$, integrating to $p^k$, and that $f := \sum_{k=1}^N f^k$ is the pdf associated with $\Prob$.
Second, the requirement $\Prob \in \mathcal{C}$ is equivalent to
\begin{align*}
    f(s) = \bar{F}(s + 1/2) - \bar{F}(s - 1/2), \quad s \in \R,
\end{align*}
for some cdf $\bar{F} \in \mathcal{F}(\R)$. In order to obtain an optimization problem over \textit{density functions} only, we represent the cdf $\bar{F}$ by its corresponding pdf $\bar{f}$; i.e., we assume that $\bar{F}$ is
differentiable almost everywhere.
Finally, by \cref{cor:QP=QhatPbar}, we can write the objective in terms of $\hat{v}$ and $\bar{f}$, rather than $v$ (or $\bar{v}$) and $f$. Together, these observations allow us to write
\begin{align*}
    \tilde{\mathcal{Q}}^{\Prob_0}_\varepsilon(x) = \sup_{f^k, \bar{f} \in \mathcal{G}_+(\R)} \Big\{ &\int_{\R} \hat{v}(s,x)\bar{f}(s)ds & & & \\
        \text{s.t.} \ \ \ &\sum_{k=1}^N \int_{\R} | s - \xi^k |^p f^k(s) ds \leq \varepsilon^p, & & \qquad [\lambda]  & \\
        &\int_{\R} f^k(s) ds = p^k, &\forall k,& \qquad  [\nu^k] & \\
        &\sum_{k=1}^N f^k(s) = \int_{s - 1/2}^{s + 1/2} \bar{f}(t)dt, &\forall s, & \qquad  [\mu(s)] & \Big\}.  
\end{align*}
The corresponding dual problem is given by 
\begin{align}
    \inf_{\lambda \in \R_+, \ \nu \in \R^N, \ \mu \in \mathcal{G}(\R)} \Big\{ &\lambda \varepsilon^p + \sum_{k=1}^N p^k \nu^k  & & & \\ \nonumber
        \text{s.t.} \ \ \ &\lambda | s - \xi^k |^p + \nu^k - \mu(s) \geq 0, &\forall s, \ \forall k & \qquad [f^k(s)]  & \\ \nonumber
        &\int_{s - 1/2}^{s + 1/2} \mu(s) ds \geq \hat{v}(s, x), &\forall s, & \qquad  [\bar{f}(s)] & \Big\}. \label{eq:Q_tilde_dual}
\end{align}
Both the primal and the dual problems have \textit{functions} as decision variables, and as a result, solving the corresponding pragmatic DRSIR problem is nontrivial, even though it is convex. In particular, the dual problem contains the function $\mu(\cdot)$, corresponding to the primal constraint $\Prob \in \mathcal{C}$. Hence, as we allude to above, the fact that we need to explicitly include the constraint $\Prob \in \mathcal{C}$ in $\tilde{\mathcal{U}}$ makes tractability difficult.

\subsubsection{Option 2: $\hat{\mathcal{U}}$} \label{subsubsec:U_hat}

Next, we consider the pragmatic uncertainty set $\hat{\mathcal{U}} := \Gamma\big( \mathcal{B}_\varepsilon(\Prob_0) \big)$. We show that this uncertainty set does lead to a tractable model. Using \cref{cor:QP=QhatPbar} we can write 
\begin{align*}
    \hat{\mathcal{Q}}^{\Prob_0}_\varepsilon(x) &= \sup_{\bar{\Prob} \in \mathcal{B}_\varepsilon(\Prob_0)} \E^{\bar{\Prob}}\big[ \hat{v}(\xi,x) \big].
\end{align*}
We can derive a dual formulation along the lines of \cref{lemma:strong_duality}.

\begin{theorem} \label{thm:Q_hat_dual}
    Consider the pragmatic Wasserstein DRSIR function $\hat{\mathcal{Q}}^{\Prob_0}_\varepsilon$. Then,
    \begin{align*}
        \hat{\mathcal{Q}}^{\Prob_0}_\varepsilon(x) &= \inf_{\lambda \in \R_+, \ \nu \in \mathcal{G}_+(S)} \Big\{ \lambda \varepsilon^p + \int_{S}  \nu(s) \Prob_0(ds) \ \Big|  \\
        &\qquad\qquad\qquad\qquad\qquad \nu(s) \geq \hat{v}(\bar{s},x) - \lambda \| s - \bar{s}\|_p^p, \quad  (s,\bar{s}) \in S \times \bar{S}  \Big\}.
    \end{align*}
\end{theorem}
\begin{proof}
    The result follows from Theorem~1 in \cite{gao2016distributionally}, which holds since $\hat{v}(\cdot,x)$ is upper semi-continuous for every $x \in \R^m$.
\end{proof}
Observe that compared with the dual problem from \cref{eq:Q_tilde_dual} corresponding to $\tilde{U}$, the dual problem of~\cref{thm:Q_hat_dual} does not have the function $\mu(\cdot)$; rather, it is replaced by $\mu(s) = \hat{v}(s,x)$. Moreover, the dual variable $\nu$ admits the solution
$\nu(s) = \sup_{\bar{s} \in \bar{S}} \big\{ \hat{v}(\bar{s},x) - \lambda \| s - \bar{s} \|_p^p \big\}$. Observing that $\hat{v}(\bar{s},x) = \max\big\{ q^+(s - x + 1/2)^+, q^+(s - x + 1/2)^+ + q^-(s - x - 1/2)^-, q^-(s - x - 1/2)^- \big\}$, we can rewrite the maximization representation of $\nu(s)$ as the maximum over three \textit{concave} maximization problems, showing that the dual problem above is indeed tractable.

We can envisage a natural row generation-based algorithm, with a master problem of the form
\begin{align*}
    \inf_{\substack{x \in X \\ \lambda \in \R_+ \\ \nu \in \mathcal{G}_+(S)}} \Big\{ c^Tx + \lambda \varepsilon^p + \int_{S}  \nu(s) \Prob_0(ds) \ \Big| \  \nu(s) \geq \hat{v}(\bar{s},x) - \lambda \| s - \bar{s}\|_p^p, \quad  (s,\bar{s}) \in \mathcal{A}  \Big\}.
\end{align*}
Every iteration, for a candidate value of $\lambda$, for each value of $s$ with positive probability mass, the constraint corresponding to $(s,\bar{s})$ is added to the set $\mathcal{A}$, where $\bar{s}$ is found by solving the separation subproblem $\sup_{\bar{s} \in \bar{S}} \big\{ \hat{v}(\bar{s},x) - \lambda \| s - \bar{s} \|_p^p \big\}$.

Finally, we show that in the special case $p=1$, the pragmatic DRSIR problem simplifies further.

\begin{theorem} \label{thm:Q_hat}
Consider the pragmatic DRSIR function $\hat{\mathcal{Q}}^{\Prob_0}_\varepsilon$ and suppose that $p=1$. Then, for every $x \in \R^m$ we have
\begin{align*}
    \hat{\mathcal{Q}}^{\Prob_0}_\varepsilon(x) = \E^{\Prob_0}\big[\hat{v}(\xi,x)\big] + \| q \|_\infty \cdot \varepsilon, 
\end{align*}
where $\hat{v}$ is the function from \cref{cor:QP=QhatPbar}.
\end{theorem}
\begin{proof}
    We first prove the result for the one-dimensional case, $m=1$. Then, we extend the result to the multi-dimensional case using an analogue of \cref{prop:separability_Q}.
    
    Let $m=1$, let $x \in \R$ be given, and suppress the $i$ index. Consider the dual problem from \cref{thm:Q_hat_dual}. Corollary~2 in \cite{gao2016distributionally} states that $\lambda = \|q\|_\infty$ is optimal for all $\varepsilon > 0$ if $\sup_{\bar{s} \in \bar{S}} \big\{ \hat{v}(\bar{s},x) - \| q \|_\infty | s - \bar{s} | \big\} \leq \hat{v}(s,x)$ for every $s \in S$. Indeed, as $\hat{v}$ is Lipschitz continuous with Lipschitz constant $L = \| q \|_\infty$, the inequality above holds with equality. Hence, $\lambda = \|q\|_\infty$ is optimal for all $\varepsilon > 0$. Substituting $\lambda = \|q\|_\infty$ and $\nu(s) = \sup_{\bar{s} \in \bar{S}} \big\{ \hat{v}(\bar{s},x) - \| q \|_\infty | s - \bar{s} | \big\} = \hat{v}(s,x)$ into the dual problem from \cref{thm:Q_hat_dual} yields
    \begin{align*}
        \hat{\mathcal{Q}}^{\Prob_0}_\varepsilon(x) &= \| q \|_\infty \varepsilon + \int_{S}  \hat{v}(s,x) \Prob_0(ds),
    \end{align*}
    which proves the result for $m=1$.
    
    Now we extend the result to the multi-dimensional case. By an analogue of \cref{prop:separability_Q}, we have
    \begin{align*}
        \hat{\mathcal{Q}}_{\varepsilon}^{\Prob_0}(x) &= \hspace{-8pt} \sup_{\varepsilon_1, \ldots, \varepsilon_m \geq 0} \Bigg\{ \sum_{i=1}^m \sup_{\Prob_i \in
        \mathcal{P}(\R)
        } \big\{  \E^{\gamma(\bar{\Prob}_i)} \big[ v_i(\xi_i,x_i)\big] \ \left | \  W_p(\Prob_{0,i}, \bar{\Prob}_i) \leq \varepsilon_i \big\} \right. \ \left | \ \sum_{i=1}^m \varepsilon_i = \varepsilon \Bigg\}. \right.
    \end{align*}
    Applying the one-dimensional result to the inner maximization problem, we obtain 
    \begin{align*}
        \hat{\mathcal{Q}}_{\varepsilon}^{\Prob_0}(x) &= \sup_{\varepsilon_1, \ldots, \varepsilon_m \geq 0} \left \{ \sum_{i=1}^m  \E^{\Prob_{0,i}} \big[ \hat{v}_i(\xi_i,x_i)\big] + \|q_i \|_\infty \cdot \varepsilon  \ \left | \ \sum_{i=1}^m \varepsilon_i = \varepsilon \right \}. \right.
    \end{align*}
    Let $i^* \in \argmax_{i=1,\ldots,m} \|q_i\|_\infty$. Then, clearly, an optimal solution is given by $\varepsilon_{i^*} = \varepsilon$ and $\varepsilon_j = 0$, $j \neq i^*$. Substituting this into the objective yields the result.
\end{proof}

By \cref{thm:Q_hat}, the pragmatic DRSIR problem with DRSIR function $\hat{\mathcal{Q}}_\varepsilon^{\Prob_0}$ is equivalent to the \textit{continuous} recourse problem, with value function $\hat{v}$, under the \textit{known} distribution $\Prob_0$. So, interestingly both difficulties of the original DRSIR problem (i.e., integer restrictions in the second stage and distributional uncertainty) disappear in this pragmatic approach. As a result, we can solve the resulting model efficiently.

\subsubsection{Distributional robustness through $\hat{v}$} \label{subsubsec:v_hat}

The results in \cref{subsubsec:U_hat} show that pragmatic DRSIR problems can be expressed in terms of the approximating value function $\hat{v}$. In fact, in the case $p=1$, \cref{thm:Q_hat} shows that the optimal solution $x^*$ does not depend on the value of $\varepsilon$, a well-known result for distributionally robust continuous recourse problems \cite{duque2020distributionally}. Hence, in that case, any robustifying behavior is a direct consequence of the fact that $v$ has been replaced by $\hat{v}$. These observations raise the question of how replacing $v$ by $\hat{v}$ can account for distributional robustness. We provide insight into this issue using the following example.

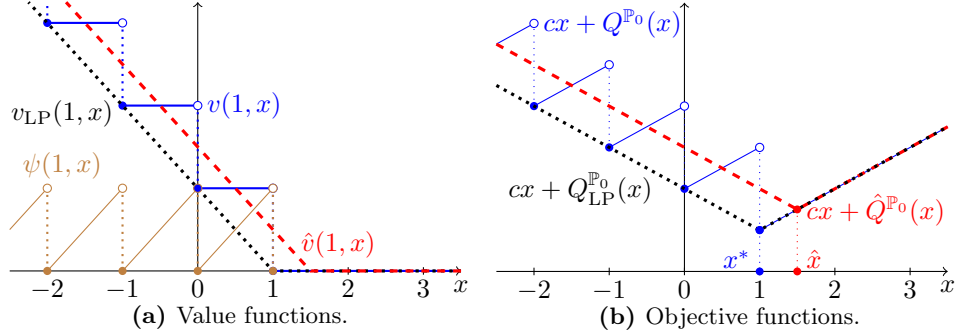
\begin{figure}[ht]
\centering
\begin{subfigure}{.5\textwidth}
  \centering
  \begin{tikzpicture}
\draw[->] (-2.5,0) -- (3.5,0);
\draw
(-2,0) node[anchor=north] {$-2$}
(-1,0) node[anchor=north] {$-1$}
(0,0) node[anchor=north] {$0$}
(1,0) node[anchor=north] {$1$}
(2,0) node[anchor=north] {$2$}
(3,0) node[anchor=north] {$3$}
(3.5,0) node[anchor=north] {$x$};
\draw 
(-2,-.05) -- (-2,0.05)
(-1,-.05) -- (-1,0.05)
(0,-.05) -- (0,0.05)
(1,-.05) -- (1,0.05)
(2,-.05) -- (2,0.05)
(3,-.05) -- (3,0.05)
;
\draw[->] (0,0) -- (0,0.55*6.5);



\draw[thick, blue] 
(3.5,0.55*0) -- (1,0.55*0)  (1,0.55*2) -- (0,0.55*2)  (0,0.55*4) -- (-1,0.55*4)  (-1,0.55*6) -- (-2,0.55*6)  (-2,0.55*6.5);
\draw[thick, dotted, blue] 
(3.5,0.55*0)  (1,0.55*0) -- (1,0.55*2)  (0,0.55*2) -- (0,0.55*4)  (-1,0.55*4) -- (-1,0.55*6)  (-2,0.55*6) -- (-2,0.55*6.5);
\draw (0,0.55*4) node[blue, anchor=west] {$v(1,x)$};
\node[circle,fill=blue, inner sep=1.1pt] at (1,0.55*0){};
\node[circle,fill=blue, inner sep=1.1pt] at (0,0.55*2){};
\node[circle,fill=blue, inner sep=1.1pt] at (-1,0.55*4){};
\node[circle,fill=blue, inner sep=1.1pt] at (-2,0.55*6){};
\node[circle, blue, fill=white, inner sep=1.1pt, draw] at (1,0.55*2){};
\node[circle, blue, fill=white, inner sep=1.1pt, draw] at (0,0.55*4){};
\node[circle, blue, fill=white, inner sep=1.1pt, draw] at (-1,0.55*6){};

\draw[very thick, dotted, black] 
(3.5,0.55*0) -- (1,0.55*0) -- (-2.25,0.55*6.5);
\draw (-1,0.55*3.8) node[black, anchor=east] {$v_{\text{LP}}(1,x)$};

\draw[very thick, dashed, red] 
(3.5,0.55*0) -- (1.5,0.55*0) -- (-1.75,0.55*6.5);
\draw (1.25,0.55*0.65) node[red, anchor=west] {$\hat{v}(1,x)$};

\draw[brown]
(3.5,0.55*0)  (1,0.55*0)  (1,0.55*2) -- (0,0.55*0)  (0,0.55*2) -- (-1,0.55*0)  (-1,0.55*2) -- (-2,0.55*0)  (-2,0.55*2) -- (-2.5,0.55*1);
\draw[thick, dotted, brown]
(3.5,0.55*0)  (1,0.55*0) -- (1,0.55*2)  (0,0.55*0) -- (0,0.55*2)  (-1,0.55*0) -- (-1,0.55*2)  (-2,0.55*0) -- (-2,0.55*2)  (-2.5,0.55*1);
\draw (-1.8,0.55*2) node[brown, anchor=south] {$\psi(1,x)$};
\node[circle,fill=brown, inner sep=1.1pt] at (-2,0.55*0){};
\node[circle,fill=brown, inner sep=1.1pt] at (-1,0.55*0){};
\node[circle,fill=brown, inner sep=1.1pt] at (0,0.55*0){};
\node[circle,fill=brown, inner sep=1.1pt] at (1,0.55*0){};
\node[circle,brown,fill=white,inner sep=1.1pt, draw] at (-2,0.55*2){};
\node[circle,brown,fill=white,inner sep=1.1pt, draw] at (-1,0.55*2){};
\node[circle,brown,fill=blue,inner sep=1.1pt, draw] at (0,0.55*2){};
\node[circle,brown,fill=white,inner sep=1.1pt, draw] at (1,0.55*2){};

\end{tikzpicture}

\vspace{-20pt}
  \caption{Value functions.}
  \label{fig:example_v_hat}
\end{subfigure}%
\begin{subfigure}{.5\textwidth}
  \centering
  \begin{tikzpicture}
\draw[->] (-2.5,0) -- (3.5,0);
\draw
(-2,0) node[anchor=north] {$-2$}
(-1,0) node[anchor=north] {$-1$}
(0,0) node[anchor=north] {$0$}
(1,0) node[anchor=north] {$1$}
(2,0) node[anchor=north] {$2$}
(3,0) node[anchor=north] {$3$}
(3.5,0) node[anchor=north] {$x$};
\draw 
(-2,-.05) -- (-2,0.05)
(-1,-.05) -- (-1,0.05)
(0,-.05) -- (0,0.05)
(1,-.05) -- (1,0.05)
(2,-.05) -- (2,0.05)
(3,-.05) -- (3,0.05)
;
\draw[->] (0,0) -- (0,0.55*6.5);





\draw[blue] 
(3.5,0.55*3.5) -- (1,0.55*1)  (1,0.55*3) -- (0,0.55*2)  (0,0.55*4) -- (-1,0.55*3)  (-1,0.55*5) -- (-2,0.55*4)  (-2,0.55*6) -- (-2.5,0.55*5.5);
\draw[dotted, blue] 
(3.5,0.55*3.5)  (1,0.55*1) -- (1,0.55*3)  (0,0.55*2) -- (0,0.55*4)  (-1,0.55*3) -- (-1,0.55*5)  (-2,0.55*4) -- (-2,0.55*6)  (-2.5,0.55*5.5);
\draw (-2,0.55*6) node[blue, anchor=west] {$cx + Q^{\Prob_0}(x)$};
\node[circle,fill=blue, inner sep=1.1pt] at (1,0.55*1){};
\node[circle,fill=blue, inner sep=1.1pt] at (0,0.55*2){};
\node[circle,fill=blue, inner sep=1.1pt] at (-1,0.55*3){};
\node[circle,fill=blue, inner sep=1.1pt] at (-2,0.55*4){};
\node[circle, blue, fill=white, inner sep=1.1pt, draw] at (1,0.55*3){};
\node[circle, blue, fill=white, inner sep=1.1pt, draw] at (0,0.55*4){};
\node[circle, blue, fill=white, inner sep=1.1pt, draw] at (-1,0.55*5){};
\node[circle, blue, fill=white, inner sep=1.1pt, draw] at (-2,0.55*6){};


\draw[dashed, very thick, red]
(-2.5,0.55*5.5) -- (1.5,0.55*1.5) -- (3.5,0.55*3.5);
\draw (1.5,0.55*1.5) node[red, anchor=west] {$cx + \hat{Q}^{\Prob_0}(x)$};

\draw[dotted, very thick, black]
(-2.5,0.55*4.5) -- (1,0.55*1) -- (3.5,0.55*3.5);
\draw (-0.3,0.55*2) node[black, anchor=east] {$cx + Q_{\text{LP}}^{\Prob_0}(x)$};

\fill[blue] (1,0.55*1) circle (1.5pt);
\draw[dotted, blue] (1,0.55*1) -- (1,0.55*0);
\fill[blue] (1,0.55*0) circle (1.5pt);
\draw (1,0.55*0.3) node[blue,anchor=east] {$x^*$};

\fill[red] (1.5,0.55*1.5) circle (1.5pt);
\draw[dotted, red] (1.5,0.55*1.5) -- (1.5,0);
\fill[red] (1.5,0.55*0) circle (1.5pt);
\draw (1.5,0.55*0.3) node[red,anchor=west] {$\hat{x}$};
\end{tikzpicture}

\vspace{-20pt}
  \caption{Objective functions.}
  \label{fig:example_Q_hat}
\end{subfigure}
\caption{A plot of the value function $v$ and related functions (a) and the corresponding objective functions (b) for the problem of \cref{ex:v_hat} with $a=1$.}
\label{fig:example_v_Q_hat}
\end{figure}

\begin{example} \label{ex:v_hat}
Consider a one-dimensional SIR problem with $c = 1$, $q^- = 0$, $q^+ = 2$; see \cref{fig:example_v_Q_hat}. Suppose we have a sample $\Xi^N = \{a\}$ of size $N=1$, and let $\Prob_0$ be the corresponding empirical distribution. If we naively assume that $\Prob_0$ is the true distribution, then we obtain an optimal value of $x^* = a$. Observe that we fit our solution $x^*$ exactly to the value of $a$, a point at which the objective function has a discontinuity (a ``jump''). Note, however, that if $\xi$ were just slightly smaller than $a$, we would have to pay the cost $q^+$ of the ``jump.'' In a context of distributional uncertainty, where the realization of $\xi$ may well be different from $a$, it follows that $x^*$ is a ``risky'' solution: it has been overfit to the sample $\Xi^N$. 

To better understand this overfitting mechanism, we decompose the value function into two parts: $v(\xi,x) = v_{\text{LP}}(\xi,x) + \psi(\xi,x)$, where $v_{\text{LP}}$ is the LP-relaxation of $v$ (i.e., the value function of the LP-relaxation of the second-stage mixed-integer program) and $\psi$ represents the ``cost of integer restrictions.'' We can split both $v_{\text{LP}}$ and $\psi(\xi,x)$ into two parts, corresponding to the two halves of their domains. First, $v_{\text{LP}}(\xi,x)$ equals $\xi - x$ if $x \leq \xi$ and zero if $x > \xi$. Second, $\psi(\xi,x)$ equals the periodic function $\psi^1(\xi,x) = \ceil{\xi - x} - (\xi - x)$ if $x \leq \xi$ and zero if $x > \xi$. Now, observe that our ``naive'' solution $x^*$ is located at a local minimum of the periodic function $\psi^1(a,x)$, right next to a ``jump.'' This is the cause of the ``riskiness'' of the solution $x^*$ observed above. We may conclude that in a distributionally uncertain context, the interaction between the periodicity of $\psi$ and the reference distribution on $\Xi^N$ leads to ``risky'' solutions that may have poor out-of-sample performance.

Given this observation, one approach to mitigate this issue may be to eliminate the periodic behavior of $\psi$. We discuss one particular approach that turns out to be especially relevant in our setting. First, we replace the periodic function $\psi^1$ by its mean value $\mu^1 := 1/2$. Next, we combine $\mu^1$ with the corresponding expression for $v_{\text{LP}}(\xi,x)$, yielding $\xi - x + \mu^1$. Doing the same for the expressions corresponding to $x > \xi$ yields the expression $0 + 0$. We combine both expressions by taking their maximum. This yields our approximating value function $\tilde{v}(\xi,x) := \max\{ \xi - x + \mu^1, 0 \} = (\xi - x + \mu^1)^+$. Interestingly, this is exactly the same as the approximating value function $\hat{v}$ from our paper: $\tilde{v} \equiv \hat{v}$. In other words, $\hat{v}$ does exactly what we set out to do: it eliminates the periodic behavior from $\psi$.

Now, applying this approach and solving the problem with the approximating value function $\hat{v}$ yields an optimal value of $x = \hat{x} := a + 1/2$, as illustrated in \cref{fig:example_v_Q_hat}. Indeed, this solution is more robust against realizations of $\xi$ that slightly differ from the observed value $a$. To make this precise, suppose that the true distribution $\Prob_{\text{true}}$ is uniformly distributed on the unit interval around $a$, i.e., on $(a - 1/2, a + 1/2)$. Observe that $\Prob_{\text{true}} = \gamma(\Prob_0)$. Then, the true optimization problem is to minimize $cx + Q^{\Prob_{\text{true}}}(x) = cx + Q^{\gamma(\Prob_0)}(x) =  cx + \hat{Q}^{\Prob_0}(x)$, where the last equality follows from \cref{lemma:QP=QhatPbar_Gamma}. Hence, in this case, the approach using $\hat{v}$ is \textit{exact}. Note that the distributional uncertainty that we hedged against in this example was about the \textit{fractional value} of $\xi$. We assume that $\xi$ is approximately located at $a$, but its fractional value is unknown. Hence, we conclude that replacing $v$ by $\hat{v}$ yields solutions that are distributionally robust against \textit{fractional} distributional uncertainty.

On a final note for this example, observe that another potential way of eliminating the periodicity of $\psi$ is to simply ignore $\psi$ altogether and replace $v$ by its LP-relaxation $v_{\text{LP}}$. However, in the example above, it is not hard to see that this approach would still yield the same ``risky'' solution $x^*$. Hence, this does not suffice to obtain a distributionally robust solution. The periodicity in $\psi$ should be eliminated, but its mean value (i.e., the ``average cost of having integer restrictions'') should be taken into account. This suggests that the results above are nontrivial. \qedexample
\end{example}

Interestingly, the discussion above provides a potential pathway for generalizing our pragmatic approach to settings beyond SIR. Recall that an analogue of the set~$\mathcal{C}$, the foundation for our approach, is not known in more general settings than SIR. However, generalizations of $\hat{v}$ is well-known for general MIR models \cite{romeijnders2016general}, which rely on the exact same process of eliminating periodicity from the function $\psi$. Hence, replacing $v$ by $\hat{v}$ could be a fruitful way to hedge against (fractional) distributional uncertainty in general mixed-integer recourse models. What is more, it leads to \textit{convex} models that are computationally more tractable than standard MIR models. In this sense, it solves two problems at once. We believe that this is a promising idea that merits future research.

\section{Stability of SIR when $\bar{\Prob} \in \mathcal{C}$ and implications} \label{sec:error_bounds}
In this section we derive a stability result for SIR, based on our analysis of pragmatic Wasserstein DRSIR problems in \cref{sec:wasserstein_DRSIR,sec:pragmatic_Wass_DRSIR}. In particular, we derive an upper bound of the form
\begin{align}
\sup_{\Prob \in \mathcal{B}_\varepsilon(\bar{\Prob})} \| Q^\Prob - Q^{\bar{\Prob}} \|_\infty  &\leq G(\varepsilon), \label{eq:bound_goal}
\end{align}
for $\bar{\Prob} \in \mathcal{C}$ if $p=1$. Interestingly, we find an upper bound $G(\varepsilon)$ that converges to zero as $\varepsilon \downarrow 0$. This is nontrivial because for, say, a degenerate distribution $\bar{\Prob}$, it is not hard to show that the left-hand side in \cref{eq:bound_goal} is at least $\| q \|_\infty$ for every $\varepsilon > 0$.

Our result can be interpreted and used in several ways. First, it is a stability result for SIR models \textit{without} distributional uncertainty because it shows that small perturbations of $\bar{\Prob} \in \mathcal{C}$ only lead to small changes in the expected value function $Q^{\bar{\Prob}}$.  Second, we use the result to compare pragmatic Wasserstein DRSIR with standard Wasserstein DRSIR. Finally, it yields an error bound for convex approximations of standard SIR models when we interpret $\Prob$ as the true distribution in the SIR model and $\bar{\Prob} \in \mathcal{C}$ as a ``convexifying'' approximating distribution in the sense made precise in \cref{def:convex_approximations} below.

\subsection{Stability result}

Our aim is to derive an inequality of the form \cref{eq:bound_goal}. We do so by finding a function, $G(\varepsilon)$, such that for every $x \in \R^m$,
\begin{subequations}\label{eq:bound_sup_inf_inequalities}
\begin{align}
\mathcal{Q}_\varepsilon^{\bar{\Prob}}(x) = \sup_{\Prob \in \mathcal{B}_\varepsilon(\bar{\Prob})} \E^{\Prob}\big[v(\xi,x)\big] &\leq \E^{\bar{\Prob}}\big[v(\xi,x)\big] + G(\varepsilon), \qquad \text{and}\label{eq:bound_sup_inequality}\\
\inf_{\Prob \in \mathcal{B}_\varepsilon(\bar{\Prob})} \E^{\Prob}\big[v(\xi,x)\big] &\geq \E^{\bar{\Prob}}\big[v(\xi,x)\big] - G(\varepsilon). \label{eq:bound_inf_inequality}
\end{align}
\end{subequations}
We focus on the first inequality \cref{eq:bound_sup_inequality}; the second inequality \cref{eq:bound_inf_inequality} follows analogously. As in \cref{prop:separability_Q} we can separate $\mathcal{Q}_\varepsilon^{\bar{\Prob}}(x)$ into $m$ one-dimensional subproblems:
\begin{align}
    \mathcal{Q}^{\bar{\Prob}}_\varepsilon(x) &= \sup_{\varepsilon_1, \ldots, \varepsilon_m \geq 0} \Bigg\{ \sum_{i=1}^m \sup_{\Prob_i \in \mathcal{P}(\R)} \Big\{  \E^{\Prob_i}\big[v_i(\xi_i,x_i)\big] \ \left | \ W_1(\bar{\Prob}_i, \Prob_i) \leq \varepsilon_i \Big\} \right. \ \left | \ \sum_{i=1}^m \varepsilon_i = \varepsilon \Bigg\}. \right. \label{eq:Q_P_bar_separable}
\end{align}
Hence, we can start with the one-dimensional case.

Let $m=1$, fix $\bar{\Prob} \in \mathcal{C}$, suppress the $i$ index, fix $x \in \R$, and consider $\mathcal{Q}^{\bar{\Prob}}_\varepsilon(x)$. In order to find an upper bound on $\mathcal{Q}^{\bar{\Prob}}_\varepsilon(x)$, it suffices to find a feasible solution to the dual problem
\begin{align*}
\inf_{\nu \in \mathcal{G}(S), \, \lambda \in \R_+} \Big\{ \int_{S} \nu(s) \bar{\Prob}(ds) + \lambda \varepsilon \ \left | \ \nu(s) + \lambda |s - \bar{s}| \geq v(\bar{s}, x), \ (s,\bar{s}) \in S \times \bar{S} \Big\}. \right.
\end{align*}
More specifically, we derive an \textit{upper bound} on the dual objective value corresponding to a dual feasible solution. By the development that led to~\cref{eq:def_nu} we can write the dual objective as 
\begin{align}
    \E^{\bar{\Prob}}\big[ v(\xi,x) \big] + \E^{\bar{\Prob}}\big[ r^\lambda(\xi,x) \big] + \lambda \varepsilon, \label{eq:dual_objective_r}
\end{align}
where $r^\lambda$ is the univariate version of the function defined in \cref{eq:def_r_plus}--\cref{eq:def_r_minus} (also see \cref{fig:nu_star}). We first derive a uniform upper bound on $\E^{\bar{\Prob}}\big[ r^\lambda(\xi,x) \big]$, using that $\bar{\Prob}$ is in~$\mathcal{C}$, and thus the probability distribution $\bar{\Prob}$ can be interpreted as a convex combination of uniform distributions on unit intervals; see \cref{prop:properties_gamma}.

\begin{lemma} \label{lemma:r_ub}
    Consider the function $r^\lambda$ from \cref{eq:def_r_plus}--\cref{eq:def_r_minus} in the one-dimensional setting. Let $x \in \R$ be given. Then, for any $\bar{\Prob} \in \mathcal{C}$ and $\lambda \geq \| q \|_\infty$ we have
    \begin{align*}
        \E^{\bar{\Prob}}\big[ r^\lambda(\xi,x) \big] &\leq \frac{ \| q \|_\infty^2}{2\lambda}.
    \end{align*}
\end{lemma}
\begin{proof}
    Since $\bar{\Prob} \in \mathcal{C}$, there exists $\tilde{\Prob}$ such that $\bar{\Prob} = \gamma(\tilde{\Prob})$. Using this fact and the fact that by part~\eqref{prop:properties_gamma:xi_integral} of \cref{prop:properties_gamma} the pdf of $\bar{\Prob}$ can be written as a convex combination of pdfs $u_s$ of uniform distributions on unit intervals $(s-1/2,s+1/2)$, $s \in \R$, we have
    \begin{align}
        \E^{\bar{\Prob}}\big[ r^\lambda(\xi,x) \big] &= \int_{-\infty}^\infty r^\lambda(s,x) \bar{f}(s) ds \nonumber = \int_{-\infty}^\infty r^\lambda(s,x) \int_{-\infty}^\infty u_t(s) d\tilde{F}(t) ds  \nonumber \\
        &= \int_{-\infty}^\infty \int_{t - 1/2}^{t + 1/2} r^\lambda(s,x) ds d\tilde{F}(t), \label{eq:expectation_r}
    \end{align}
    where $\bar{f}$ is the pdf of $\xi$ under $\bar{\Prob}$, $\tilde{F}$ is the cdf of $\xi$ under $\tilde{\Prob}$, and where the last equality follows from Fubini's theorem. In order to find an upper bound on this expression, we first derive an upper bound on $R^\lambda(t,x) := \int_{t - 1/2}^{t + 1/2} r^\lambda(s,x) ds$, for any $t \in \R$. We note that \cref{fig:nu_star} is helpful for the remainder of the proof.
    
     Fix $x \in \R$ and without loss of generality, suppose that $q^+ \geq q^-$, so $r^\lambda$ is defined as in \cref{eq:def_r_plus}. Write $\rho^\lambda_1(s,x) := (q^- - \lambda(s - \floor{s}_x))^+$, $\rho^\lambda_2(s,x) := (q^+ - q^- - \lambda( \ceil{s}_x - s))^+$, and $\rho^\lambda_3(s,x) := (q^+ - \lambda( \ceil{s}_x - s))^+$, $s,x \in \R$, for the functions in \cref{eq:def_r_plus}. Observe that $\rho^\lambda_j(s,x)$ is periodic in $s$ with period 1, $j=1,2,3$. Furthermore, we can identify when each of these functions equal zero:
     \begin{align}
         \rho^\lambda_1(s,x) &= 0 \iff s \in [x + k + q^-/\lambda, x + k + 1) \text{ for some } k \in \Z, \label{eq:rho_1} \\
         \rho^\lambda_2(s,x) &= 0 \iff s \in (x + k - 1, x + k - (q^+ - q^-)/\lambda] \text{ for some } k \in \Z, \label{eq:rho_2} \\
         \rho^\lambda_3(s,x) &= 0 \iff s \in (x + k - 1, x + k - q^+/\lambda] \text{ for some } k \in \Z. \label{eq:rho_3}
     \end{align}
     We use these properties to show that $R^\lambda(t,x) \leq \frac{(q^+)^2}{2\lambda}$ for all $t \in \R$. We consider three cases for $t \in \R$.
     
     First, suppose that $t \geq x + 1/2$. Then, $r^\lambda(s,x) = \rho^\lambda_3(s,x) = (q^+ - \lambda (\ceil{s}_x - s))^+$, for all $s \in [t-1/2, t+1/2]$. Hence, 
     \begin{align}
        R^\lambda(t,x) &= \int_{t - 1/2}^{t + 1/2} (q^+ - \lambda (\ceil{s}_x - s))^+ ds = \int_{x}^{x+1} (q^+ - \lambda (\ceil{s}_x - s))^+ ds \nonumber \\
        &=\int_{x}^{x+1} (q^+ - \lambda (x + 1 - s))^+ ds = \int_0^1 (q^+ - \lambda (1 - u))^+ du = \frac{(q^+)^2}{2\lambda}, \label{eq:R_integral_1}
    \end{align}
    where we use periodicity of $\rho^\lambda_3$, the substitution $u = s - x$, and the inequality $\lambda \geq q^+$ in the second, fourth, and fifth equality, respectively.
    
    Second, suppose that $t \leq x + 1/2 - \frac{q^+}{\lambda}$. Then, by \cref{eq:def_r_plus} and \cref{eq:rho_1}--\cref{eq:rho_3} it follows that for all $s \in [t-1/2, t+1/2]$, we either have $r^\lambda(s,x) = \rho^\lambda_1(s,x) \geq 0$, or $r^\lambda(s,x) = \rho^\lambda_2(s,x) \geq 0$, or $r^\lambda(s,x) = \rho^\lambda_3(s,x) = 0$. It follows that for all $s \in [t-1/2, t+1/2]$, we have $r^\lambda(s,x) \leq \rho^\lambda_1(s,x) + \rho^\lambda_2(s,x)$. Hence,
    \begin{align*}
        R^\lambda(t,x) &\leq \int_{t-1/2}^{t+1/2} \rho^\lambda_1(s,x) ds + \int_{t-1/2}^{t+1/2} \rho^\lambda_2(s,x) ds \leq \frac{(q^-)^2}{2\lambda} + \frac{(q^+ - q^-)^2}{2\lambda} \leq \frac{(q^+)^2}{2\lambda}, 
    \end{align*}
    where the first inequality follows similarly as \cref{eq:R_integral_1}, and the second inequality follows from $(q^-)^2 + (q^+ - q^-)^2 = (q^+)^2 + 2(q^-)^2 -2q^+ q^- = (q^+)^2 + 2q^-(q^- - q^+) \leq (q^+)^2$, since $0 \leq q^- \leq q^+$ by assumption. 
    
    Finally, suppose that $x + 1/2 - \frac{q^+}{\lambda} < t < x + 1/2$. Consider the derivative of $R^\lambda(t,x)$ with respect to $t$. We have, using \cref{eq:def_r_plus} and \cref{eq:rho_1}--\cref{eq:rho_3}, that
    \begin{align*}
        \dv{}{t}R^\lambda(t,x) &= r^\lambda(t+1/2,x) - r^\lambda(t-1/2,x) = \rho^\lambda_3(t+1/2,x) - \rho^\lambda_2(t-1/2,x) \\
        &= \rho^\lambda_3(t+1/2,x) - \rho^\lambda_2(t+1/2,x) \geq 0,
    \end{align*}
    where the last equality follows by periodicity of $\rho^\lambda_2$. Hence, $\dv{}{t}R^\lambda(t,x) \geq 0$ for all $t \in (x + 1/2 - \frac{q^+}{\lambda} , x + 1/2)$. It follows that for all $t \in (x + 1/2 - \frac{q^+}{\lambda} , x + 1/2)$, we have $R^\lambda(t,x) \leq \lim_{\tau \to x + 1/2} R^\lambda(\tau,x) = R^\lambda(x+1/2,x) = \frac{(q^+)^2}{2\lambda}$.
    
    The three cases above prove that $R^\lambda(t,x) \leq \frac{(q^+)^2}{2\lambda}$ for all $t \in \R$. Combining this with \cref{eq:expectation_r}, we obtain 
    \begin{align*}
        \E^{\bar{\Prob}}\big[ r^\lambda(\xi,x) \big] \leq \int_{\R} \frac{ \| q \|_\infty^2}{2\lambda} \bar{\Prob}(ds) = \frac{ \| q \|_\infty^2}{2\lambda},
    \end{align*}
    which proves the result.
\end{proof}

By \cref{lemma:r_ub} the dual objective \cref{eq:dual_objective_r} is bounded from above by
\begin{align}
    \E^{\bar{\Prob}}\big[ v(\xi,x) \big] + \frac{ \| q \|_\infty^2}{2\lambda} + \lambda \varepsilon. \label{eq:dual_objective_r_ub}
\end{align}
This constitutes an upper bound on $\mathcal{Q}_{\varepsilon}^{\bar{\Prob}}(x)$ for any value of $\lambda \geq \| q \|_\infty$. To find the \textit{best} upper bound over $\lambda$, we minimize expression~\cref{eq:dual_objective_r_ub} over $\lambda$. Hence, we solve
\begin{align*}
\inf_{\lambda \geq \| q \|_\infty} \Big\{ \E^{\bar{\Prob}}\big[ v(\xi,x) \big] + \frac{ \| q \|_\infty^2}{2\lambda} + \lambda \varepsilon  \Big\},
\end{align*}
which yields
\begin{align*}
    \lambda^*(\varepsilon) = \begin{cases}
    \|q\|_\infty / \sqrt{2\varepsilon}, &\text{if } 0 < \varepsilon \leq 1/2, \\
    \|q\|_\infty, &\text{if } \varepsilon \geq 1/2.
    \end{cases}
\end{align*}
Plugging this into the dual objective yields an inequality of the form \cref{eq:bound_sup_inequality}. Using an analogous analysis for \cref{eq:bound_inf_inequality} yields the following result for the one-dimensional case. Here, we use $g$ rather than the $G$ in inequalities~\cref{eq:bound_sup_inf_inequalities} to distinguish it from the multivariate result which follows. 

\begin{theorem} \label{thm:Wasserstein_bound_one-dim}
Consider the one-dimensional Wasserstein DRSIR function $\mathcal{Q}^{\bar{\Prob}}_\varepsilon$ with $\bar{\Prob} \in \mathcal{C}$. Then, for every $\varepsilon > 0$, we have
\begin{align*}
    \sup_{\Prob \in \mathcal{B}_\varepsilon(\bar{\Prob})} \| Q^\Prob - Q^{\bar{\Prob}} \|_\infty  &\leq g(\varepsilon),
\end{align*}
where $g: \R_+ \to \R$ is defined by
\begin{align*}
    g(\varepsilon) = \begin{cases}
    \| q \|_\infty \sqrt{2 \varepsilon}, &\text{if } 0 \leq \varepsilon \leq 1/2, \\
    \| q \|_\infty (\varepsilon + 1/2), &\text{if } \varepsilon \geq 1/2.
    \end{cases}
\end{align*}
\end{theorem}
\begin{proof}
    The result is equivalent to the expressions in \cref{eq:bound_sup_inf_inequalities}, which follow from the discussion above.
\end{proof}

As $\lim_{\varepsilon \downarrow 0} g(\varepsilon) = 0$, \cref{thm:Wasserstein_bound_one-dim} shows that the Wasserstein SIR function is stable at any $\bar{\Prob} \in \mathcal{C}$. Intuitively, this fact can be interpreted as follows. Starting out with a distribution $\bar{\Prob} \in \mathcal{C}$, we have a Wasserstein budget of $\varepsilon$ to spend on changing the distribution in such a way that we increase the objective function as much as possible. This gain in the objective is indicated by the function~$g$; see \cref{fig:g}. Suppose for the moment that $0 < \varepsilon < 1/2$. Then, the most profitable changes in the distribution are ``rounding'' moves, i.e., moving probability mass over a distance less than one to points in $x + \Z$, where the value function $v(\cdot,x)$ has jumps. We make these rounding moves in decreasing order of (adversarial) profitability, hence the decreasing slope of $g$ in the figure. It takes a budget of at most $\varepsilon = 1/2$ to complete such rounding moves, and any additional budget is spent on moves of integer size, keeping all probability mass on $x + \Z$. Probability mass is moved in the most profitable direction, i.e., with a marginal return of $\|q\|_\infty$, which is captured in the slope of~$g$.
%
The following example shows that the inequality of \cref{thm:Wasserstein_bound_one-dim} is tight in the sense that it can hold with equality. 

\begin{example}\label{ex:tightness}
    Consider the setting of \cref{thm:Wasserstein_bound_one-dim} and suppose that $q^+ \geq q^-$, $\varepsilon = 1/2$, and let $\bar{\Prob}$ represent a uniform distribution on $(a,a+1)$ for some $a \in\R$. Let $\Prob^*$ denote a degenerate distribution at $a+1$. Note that $W_1(\bar{\Prob}, \Prob^*) = 1/2 = \varepsilon$, so $\Prob^* \in \mathcal{B}_{\varepsilon}(\bar{\Prob})$. Hence,
    \begin{align*}
        \sup_{\Prob \in \mathcal{B}_\varepsilon(\bar{\Prob})} \| Q^\Prob - Q^{\bar{\Prob}} \|_\infty &\geq \sup_{\Prob \in \mathcal{B}_\varepsilon(\bar{\Prob})} \big\{ Q^{\Prob}(a) - Q^{\bar{\Prob}}(a) \big\} = \sup_{\Prob \in \mathcal{B}_\varepsilon(\bar{\Prob})} \E^{\Prob}\big[\bar{v}(\xi,a)\big] - \E^{\bar{\Prob}}\big[\bar{v}(\xi,a)\big] \\
        &\geq \E^{\Prob^*}\big[\bar{v}(\xi,a)\big] - \E^{\bar{\Prob}}\big[\bar{v}(\xi,a)\big] = \bar{v}(a+1,a) - \int_{a}^{a+1} \bar{v}(s,a) ds \\
        &= 2q^+ - q^+ = q^+ = g(\varepsilon),
    \end{align*}
    where the first equality follows by an analogue of \cref{lemma:v_usc} and continuity of the uniform distribution $\bar{\Prob}$. Combining this inequality with \cref{thm:Wasserstein_bound_one-dim}, we obtain $\sup_{\Prob \in \mathcal{B}_\varepsilon(\bar{\Prob})} \| Q^\Prob - Q^{\bar{\Prob}} \|_\infty = g(\varepsilon)$. \qedexample
\end{example}

\begin{figure}
    \centering
    \begin{tikzpicture}
  \draw[->] (-.5, 0.7*0) -- (6.5, 0.7*0) node[right] {$\varepsilon$};
  \draw[->] (0, 0.7*0) -- (0, 0.7*5.5);
  \draw (0,0) node[anchor=north] {$0$};
  \draw[scale=3, domain=0:1, smooth, thick, variable=\x, blue, samples=100] plot ({\x}, {0.7*(\x)^(1/2)});
  \draw[scale=3, domain=1:2, smooth, thick, variable=\x, blue] plot ({\x}, {0.7*(0.5*\x + 0.5)}) node[blue,anchor=west]{$g(\varepsilon)$};
  \draw[dotted, thick] (3,0.7*3) -- (3,0.7*0) node[anchor=north] {$1/2$};
  \draw[dotted, thick] (3,0.7*3) -- (0,0.7*3) node[anchor=east] {$\|q\|_\infty$};
  \draw[dashed] (0,0.7*0) -- (6,0.7*3) node[anchor=west] {$\| q \|_\infty \varepsilon$};
  \draw[<->, red, thick] (4,0.7*2) -- (4,0.7*3.5);
  \draw (4,0.7*3) node[anchor=west, red] {$g^*(\varepsilon)$};
\end{tikzpicture}

\vspace{-20pt}
    \caption{The function $g$ from \cref{thm:Wasserstein_bound_one-dim} and the one-dimensional analogue $g^*$ of the function $G^*$ from \cref{thm:DRSIR_bounds}.}
    \label{fig:g}
\end{figure}
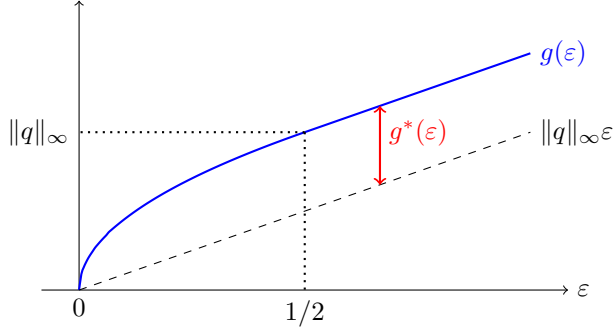

We can extend \cref{thm:Wasserstein_bound_one-dim} to the multi-dimensional case by applying \cref{thm:Wasserstein_bound_one-dim} to the inner maximization problems in \cref{eq:Q_P_bar_separable} and solving the outer budget allocation problem over $\varepsilon_i$, $i=1,\ldots,m$.

\begin{theorem} \label{thm:Wasserstein_bound_multi-dim}
Consider the multi-dimensional Wasserstein DRSIR function $\mathcal{Q}^{\bar{\Prob}}_\varepsilon$ with $\bar{\Prob} \in \mathcal{C}$. Then, for every $\varepsilon > 0$ we have
\begin{align*}
    \sup_{\Prob \in \mathcal{B}_\varepsilon(\bar{\Prob})} \| Q^\Prob - Q^{\bar{\Prob}} \|_\infty  &\leq G(\varepsilon),
\end{align*}
where $G: \R_+ \to \R$ is defined by
\begin{align}
    G(\varepsilon) := \begin{cases}
    \|\bar{q}\|_2 \sqrt{ 2 \varepsilon} &\text{if } 0 \leq \varepsilon \leq \bar{\varepsilon}, \\
    \| \bar{q} \|_2 \sqrt{2 \bar{\varepsilon}} + \| \bar{q} \|_\infty (\varepsilon - \bar{\varepsilon}) &\text{if } \varepsilon > \bar{\varepsilon},
    \end{cases} \label{eq:def_G}
\end{align}
where $\bar{q} := (\bar{q}_1,\ldots,\bar{q}_m)$, with  $\bar{q}_i := \| q_i \|_\infty = \max\{q^+_i, q^-_i\}$, and $\bar{\varepsilon}=\frac{1}{2} \| \bar{q} \|_2^2 / \| \bar{q} \|_\infty^2$.
\end{theorem}
\begin{proof}
    It suffices to show that 
    \begin{subequations}
    \begin{align}
        \mathcal{Q}^{\bar{\Prob}}_\varepsilon(x) = \sup_{\Prob \in \mathcal{B}_\varepsilon(\bar{\Prob})} \E^{\Prob}\big[v(\xi,x)\big] &\leq \E^{\bar{\Prob}}\big[ v(\xi,x) \big] + G(\varepsilon), \quad \text{and} \label{eq:theorem_G_sup}\\
        \inf_{\Prob \in \mathcal{B}_\varepsilon(\bar{\Prob})} \E^{\Prob}\big[v(\xi,x)\big] &\geq \E^{\bar{\Prob}}\big[ v(\xi,x) \big] - G(\varepsilon). \label{eq:theorem_G_inf}
    \end{align}
    \end{subequations}
    We prove that \cref{eq:theorem_G_sup} holds; \cref{eq:theorem_G_inf} follows analogously. Applying \cref{thm:Wasserstein_bound_one-dim} to each of the inner maximization problems in \cref{eq:Q_P_bar_separable} yields    
    \begin{align*}
        \mathcal{Q}^{\bar{\Prob}}_\varepsilon(x) &\leq \sup_{\varepsilon_1, \ldots, \varepsilon_m \geq 0} \left \{ \left. \sum_{i=1}^m \Big( \E^{\bar{\Prob}_i}\big[v(\xi_i, x_i)\big] + g_i(\varepsilon_i) \Big)\ \right | \ \sum_{i=1}^m \varepsilon_i = \varepsilon \right \}  \\
        &=  \E^{\bar{\Prob}}\big[v(\xi, x)\big] + \sup_{\varepsilon_1, \ldots, \varepsilon_m \geq 0} \left \{ \sum_{i=1}^m g_i(\varepsilon_i) \ \left | \ \sum_{i=1}^m \varepsilon_i = \varepsilon \right \}, \right.
    \end{align*}
    where $g_i$ denotes the function $g$ from \cref{thm:Wasserstein_bound_one-dim} for the one-dimensional DRSIR model corresponding to the $i^{th}$ dimension. Now, observing that $g_i$, $i=1,\ldots,m$, are non-decreasing, concave, continuously differentiable (for $\varepsilon_i > 0$) functions, it follows from the first-order conditions that in the optimum, we must have $g^\prime_i(\varepsilon_i) = g^\prime_j(\varepsilon_j)$, $\forall i,j = 1,\ldots,m$.
    If $0 < \varepsilon_i, \varepsilon_j \leq 1/2$, then
    \begin{align*}
        g_i^\prime(\varepsilon_i) = g_j^\prime(\varepsilon_j) & \iff \frac{\bar{q}_i}{\sqrt{2 \varepsilon_i}} = \frac{\bar{q}_j}{\sqrt{2 \varepsilon_j}} \iff \varepsilon_j = \Big(\frac{\bar{q}_j}{\bar{q}_i}\Big)^2 \varepsilon_i.
    \end{align*}
    Substituting this into $\sum_{j=1}^m \varepsilon_j = \varepsilon$ for every $j=1,\ldots,m$,
    yields
    $\varepsilon_i = \frac{\bar{q}_i^2}{\|\bar{q}\|_2^2} \, \varepsilon$,
    $i=1,\ldots,m$. The corresponding total value is
    \begin{align*}
        \sum_{i=1}^m g_i(\varepsilon_i) = \sum_{i=1}^m \bar{q}_i \sqrt{ 2 \frac{\bar{q}_i^2}{\|\bar{q}\|_2^2} \varepsilon} =
        \|\bar{q}\|_2 \sqrt{ 2 \varepsilon}.
    \end{align*}
    This value is maximal if and only if $\varepsilon_j \leq 1/2$ for all $j=1,\ldots,m$, which holds if $\frac{\bar{q}_j^2}{\|\bar{q} \|_2} \varepsilon \leq 1/2$, for all $j=1,\ldots,m$, which is equivalent to $\varepsilon \leq \frac{1}{2} \left( \frac{\| \bar{q} \|_2}{\| \bar{q} \|_\infty} \right)^2 =: \bar{\varepsilon}$.
    
    In case $\varepsilon > \bar{\varepsilon}$, then an optimal solution is to set $\varepsilon_i > 1/2$, so that $g^\prime(\varepsilon_i) = \bar{q}_i$, for an index $i$ with $\bar{q}_i = \|\bar{q} \|_\infty$, and for indices $j\neq i$ set $\varepsilon_j \leq 1/2$ in such a way that $g_j^\prime(\varepsilon_j) = g_i^\prime(\varepsilon_i) = \| \bar{q} \|_\infty$. This holds if and only if $\varepsilon_j = \frac{1}{2} \left( \frac{\bar{q}_j}{\|q\|_\infty} \right)^2$ for all $j\neq i$. Since $\varepsilon_i = \varepsilon - \sum_{j \neq i} \varepsilon_j$, it follows that
    \begin{align*}
        \varepsilon_i = \varepsilon - \sum_{j \neq i} \frac{1}{2} \left( \frac{\bar{q}_j}{\|q\|_\infty} \right)^2 = \varepsilon - \frac{1}{2} \sum_{j=1}^m \left( \frac{\bar{q}_j}{\|q\|_\infty} \right)^2 + \frac{1}{2} \left( \frac{\bar{q}_i}{\|q\|_\infty} \right)^2 = \varepsilon - \bar{\varepsilon} + 1/2,
    \end{align*}
    and thus we obtain a total value of
    \begin{align*}
        \sum_{j=1}^m g_j(\varepsilon_j) &= \sum_{j \neq i} g_j \left(\frac{1}{2} \left( \frac{\bar{q}_j}{\|q\|_\infty} \right)^2\right) + g_i(\varepsilon - \bar{\varepsilon} + 1/2) \\
        &= \sum_{j \neq i} \bar{q}_j \sqrt{\left( \frac{\bar{q}_j}{\| \bar{q} \|_\infty} \right)^2} + \| \bar{q} \|_\infty (\varepsilon - \bar{\varepsilon} + 1) = \frac{\| \bar{q} \|_2^2}{\| \bar{q} \|_\infty} + \| \bar{q} \|_\infty (\varepsilon - \bar{\varepsilon}).
    \end{align*}
    The desired result now follows from the fact that $\frac{\| \bar{q} \|_2^2}{\| \bar{q} \|_\infty} = \|\bar{q} \|_2 \sqrt{2 \bar{\varepsilon}}$.
\end{proof}
The interpretation of \cref{thm:Wasserstein_bound_multi-dim} is analogous to that of \cref{thm:Wasserstein_bound_one-dim}. The only difference is that we now have multiple directions (dimensions) in which we can move probability mass and we pick the most profitable combination of those. 
In the next two subsections we show the implications of \cref{thm:Wasserstein_bound_multi-dim} in two different settings.

\subsection{Comparing standard and pragmatic Wasserstein DRSIR} \label{subsec:comparing}

Here, we use \cref{thm:Wasserstein_bound_multi-dim} to compare standard and pragmatic Wasserstein DRSIR in terms of the objective value. That is, we quantify how much power nature loses if we restrict ourselves to distributions in $\mathcal{C}$.

\begin{theorem} \label{thm:DRSIR_bounds}
    Consider the DRSIR functions $\mathcal{Q}_\varepsilon^{\Prob_0}$ and $\hat{\mathcal{Q}}_\varepsilon^{\Prob_0}$ from \cref{eq:Wass_DRSIR_problem_naive} and \cref{eq:def_Q_tilde_and_hat}. Suppose that $\varepsilon \geq W_1(\Prob_0, \hat{\Prob})$, where $\hat{\Prob} := \Gamma(\Prob_0)$, with $\Gamma$ given in \cref{def:Gamma}. Then, for every $x \in \R^m$,
    \begin{align}
        \hat{\mathcal{Q}}_\varepsilon^{\Prob_0}(x) - \|q\|_\infty &W_1(\Prob_0, \hat{\Prob}) \leq \mathcal{Q}_\varepsilon^{\Prob_0}(x) \label{eq:DRSIR_bounds} \\ &\leq 
        \hat{\mathcal{Q}}_\varepsilon^{\Prob_0}(x) + G^*(\varepsilon + W_1(\Prob_0, \hat{\Prob})) + \|q\|_\infty \cdot W_1(\Prob_0, \hat{\Prob}), \nonumber
    \end{align}
    where $G^* : \R_+ \to \R$ is defined by
    \begin{align*}
        G^*(\varepsilon) = G(\varepsilon) - \|q\|_\infty \cdot \varepsilon, \quad \varepsilon \in \R_+.
    \end{align*}
\end{theorem}
\begin{proof}
    We start by proving the first inequality in \cref{eq:DRSIR_bounds}. Consider a feasible solution $\Prob$ to the DRO problem defined by $\mathcal{Q}_{\varepsilon - W_1(\Prob_0,\hat{\Prob})}^{\hat{\Prob}}(x)$. Then, since $W_1(\Prob,\Prob_0) \leq W_1(\Prob, \hat{\Prob}) + W_1(\hat{\Prob}, \Prob_0) \leq \varepsilon - W_1(\hat{\Prob}, \Prob_0) + W_1(\hat{\Prob}, \Prob_0) = \varepsilon$, $\Prob$ is also feasible for the DRO problem defined by $\mathcal{Q}_\varepsilon^{\Prob_0}(x)$. Hence, $\mathcal{Q}_\varepsilon^{\Prob_0}(x) \geq \mathcal{Q}_{\varepsilon - W_1(\Prob_0,\hat{\Prob})}^{\hat{\Prob}}(x)$. Starting from $\hat{\Prob}$, we can construct a sequence of feasible solutions to $\mathcal{Q}_{\varepsilon - W_1(\Prob_0,\hat{\Prob})}^{\hat{\Prob}}(x)$ with an objective value converging to $\E^{\hat{\Prob}}\big[v(\xi,x)\big] + \|q\|_\infty \cdot \big( \varepsilon - W_1(\Prob_0,\hat{\Prob}) \big)$. By \cref{thm:Q_hat} and \cref{lemma:QP=QhatPbar_Gamma} it follows that $\mathcal{Q}_\varepsilon^{\Prob_0}(x) \geq \mathcal{Q}_{\varepsilon - W_1(\Prob_0,\hat{\Prob})}^{\hat{\Prob}}(x) \geq \hat{\mathcal{Q}}_\varepsilon^{\Prob_0}(x) - \|q\|_\infty \cdot W_1(\Prob_0,\hat{\Prob})$.
    
    For the upper bound in \cref{eq:DRSIR_bounds}, consider a feasible solution $\Prob$ for $\mathcal{Q}_\varepsilon^{\Prob_0}(x)$. Then, since $W_1(\hat{\Prob}, \Prob) \leq W_1(\hat{\Prob}, \Prob_0) + W_1(\Prob_0, \Prob) \leq W_1(\hat{\Prob}, \Prob_0) + \varepsilon$, $\Prob$ is also a feasible solution for $\mathcal{Q}_{\varepsilon + W_1(\Prob_0,\hat{\Prob})}^{\hat{\Prob}}(x)$. Hence, $\mathcal{Q}_\varepsilon^{\Prob_0}(x) \leq \mathcal{Q}_{\varepsilon + W_1(\Prob_0,\hat{\Prob})}^{\hat{\Prob}}(x)$. Since $\hat{\Prob} \in \mathcal{C}$, it follows from \cref{thm:Wasserstein_bound_multi-dim} that
    \begin{align*}
        \mathcal{Q}_\varepsilon^{\Prob_0}(x) &\leq \mathcal{Q}_{\varepsilon + W_1(\Prob_0,\hat{\Prob})}^{\hat{\Prob}}(x) \leq \E^{\Prob_0}\big[v(\xi,x)\big] + G\big( \varepsilon + W_1(\Prob_0,\hat{\Prob}) \big) \\
        &= \hat{\mathcal{Q}}_\varepsilon^{\Prob_0}(x) + G^*\big(\varepsilon + W_1(\Prob_0, \hat{\Prob}) \big) + \|q\|_\infty \cdot W_1(\Prob_0, \hat{\Prob}).
    \end{align*}
\end{proof}

To interpret \cref{thm:DRSIR_bounds}, suppose that $W_1(\Prob_0, \Gamma(\Prob_0))$ is negligibly small compared to $\varepsilon$. Then, $\mathcal{Q}_\varepsilon^{\Prob_0}(x) - \hat{\mathcal{Q}}_\varepsilon^{\Prob_0}(x)$ is bounded from above by $G^*(\varepsilon)$. This function quantifies the difference between exploiting the ``jumps'' in $v(\cdot,x)$ and prohibiting this, and only allowing a ``linear'' reward $\|q\|_\infty \cdot \varepsilon$; see \cref{fig:g} for an illustration. The latter is exactly what happens when we restrict ourselves to distributions in $\mathcal{C}$.

\subsection{Error bounds for convex approximations of SIR} \label{subsec:error_bounds}

A final interpretation of \cref{thm:Wasserstein_bound_multi-dim} is in terms of a performance guarantee for so-called \textit{convex approximations} of SIR models, under a known distribution; i.e., this subsection is of interest outside of distributionally robust optimization. Convex approximations have been proposed in the literature as a means to overcome non-convexity of MIR models \cite{romeijnders2016general,romeijnders2015,romeijnders2016tu,vanbeesten2020convex,vanderlaan2020LBDA,vlerk2004,vlerk2010}. The idea is to approximate the (typically non-convex) recourse function $Q^{\Prob}(x) = \E^{\Prob}\big[v(\xi,x)\big]$ by a convex function. Then, the approximating model is convex and can be solved efficiently.

In order to guarantee the performance of these convex approximations, significant effort has been made in the literature to derive \textit{error bounds}: upper bounds on the approximation error \cite{romeijnders2017assessing,romeijnders2016general,romeijnders2015,romeijnders2016tu,vanbeesten2020convex,vanbeesten2022parametric,vanderlaan2020LBDA,vanderlaan2018higher}. This has led to substantial progress: the most general results are obtained for general MIR models with randomness in $h$, $q$, and $T$, and with a mean-CVaR objective function \cite{vanbeesten2020convex}. However, all error bounds in the literature depend on the total variation of the one-dimensional conditional probability density functions of $h$ and hence, rely on the crucial assumption that the distribution of the second-stage right-hand side vector $h$ is \textit{continuous}. No non-trivial error bounds are known that hold for discrete distributions. 

Here we derive an alternative error bound for convex approximations of the form $Q^{\tilde{\Prob}}$, $\tilde{\Prob} \in \mathcal{C}$, based on the Wasserstein distance between the original distribution $\Prob$ and the ``convexifying'' distribution $\tilde{\Prob}$. This error bound relies heavily on our ability to construct convex approximations of by use of a ``convexifying'' distribution $\tilde{\Prob} \in \mathcal{C}$. This is in sharp contrast with the literature, where convex approximations are typically defined in terms of approximating the \textit{value function} $v$ by a convex function~$\tilde{v}$. As we have observed before, transforming the distribution or transforming the value function are equivalent for SIR models, but the distribution perspective is particularly potent in this case. 

\begin{theorem} \label{thm:Wass_error_bound}
Consider the SIR function $Q^\Prob$ for some $\Prob \in \mathcal{P}(\R^m)$. Let $\tilde{\Prob} \in \mathcal{C}$ be given. Then,
\begin{align*}
    \| Q^{\Prob} - Q^{\tilde{\Prob}} \|_\infty \leq G(W_1(\Prob, \tilde{\Prob})),
\end{align*}
where $G$ is defined as in \cref{eq:def_G}.
\end{theorem}
\begin{proof}
 The proof follows directly from \cref{thm:Wasserstein_bound_multi-dim}.
\end{proof}

\Cref{thm:Wass_error_bound} shows that a convex approximation $Q^{\tilde{\Prob}}$ of $Q^\Prob$ is good if the Wasserstein distance between the original distribution $\Prob$ and the approximating distribution $\tilde{\Prob}$ is small. Importantly, in contrast with the total variation error bounds from the literature, the error bound from \cref{thm:Wass_error_bound} requires no continuity assumption on the original distribution $\Prob$ of $h$. In fact, it holds for \textit{any} distribution $\Prob \in \mathcal{P}(\R^m)$.  

\Cref{thm:Wass_error_bound} has implications for two popular convex approximations from the literature: the \textit{shifted LP-relaxation approximation} and the \textit{$\alpha$-approximation}. These two approximations have been studied extensively in the relevant literature, both in the SIR setting and in more general MIR settings. Given their importance, we explicitly present the implications of \cref{thm:Wass_error_bound} for these two approximations.

\begin{definition} \label{def:convex_approximations}
Consider the SIR function $Q^\Prob$ and let $F$ be the joint cdf of $\xi$ under $\Prob$. Then, the \textit{shifted LP-relaxation approximation} of $Q^{\Prob}$ is given by $Q^{\hat{\Prob}}$, where $\hat{\Prob}$ is the distribution with corresponding joint cdf $\hat{F} := \Gamma(F)$ with $\Gamma$ as in \cref{def:Gamma}. Moreover, for $\alpha \in \R^m$, the \textit{$\alpha$-approximation} of $Q^\Prob$ is given by $Q^{\tilde{\Prob}_\alpha}$, where $\tilde{\Prob}_\alpha$ is the distribution with corresponding cdf $\tilde{F}_\alpha := \Gamma_\alpha(F)$, where $\Gamma_\alpha : \mathcal{F}(\R^m) \to \mathcal{F}(\R^m)$ is defined by
\begin{align*}
    \Gamma_\alpha \circ F (s) = \int_{\floor{s_1}_{\alpha_1}}^{\floor{s_1}_{\alpha_1} + 1} \cdots \int_{\floor{s_m}_{\alpha_m}}^{\floor{s_m}_{\alpha_m} + 1} F(t) dt_m \cdots dt_1, \quad s \in \R^m,
\end{align*}
where $\ceil{t}_a := \ceil{t - a} + a$ and $\floor{t}_a := \floor{t - a} + a$, for all $t, a \in \R$. Slightly abusing notation, we write $\hat{\Prob} = \Gamma(\Prob)$ and $\tilde{\Prob}_\alpha = \Gamma_\alpha(\Prob)$.
\end{definition}

\begin{corollary}\label{cor:error_bounds_concrete}
Consider the SIR function $Q^\Prob$ for some $\Prob \in \mathcal{P}(\R^m)$ and its convex approximations from \cref{def:convex_approximations}. Let $\hat{\Prob} := \Gamma(\Prob)$ and $\tilde{\Prob}_\alpha := \Gamma_\alpha(\Prob)$. Then,
\begin{align*}
    \| Q^{\Prob} - Q^{\hat{\Prob}} \|_\infty \leq G(W_1(\Prob, \hat{\Prob})), \quad \text{and} \quad  \| Q^{\Prob} - Q^{\tilde{\Prob}_\alpha} \|_\infty \leq G(W_1(\Prob, \tilde{\Prob}_\alpha)).
\end{align*}
\end{corollary}
\begin{proof}
    The proof follows immediately from \cref{thm:Wass_error_bound} since $\hat{\Prob}, \, \tilde{\Prob}_\alpha \in \mathcal{C}$.
\end{proof}

An interesting question is how the Wasserstein error bounds from \cref{cor:error_bounds_concrete} compare with the total variation error bounds from the literature. To answer this question, we repeat the best known error bounds from the literature in \cref{lemma:totvar_bound} below and subsequently explore an example that provides insight into the relative performance between both types of error bounds.

\begin{lemma}\label{lemma:totvar_bound}
Consider the SIR function $Q^\Prob$ from 
\cref{eq:separability}
and its convex approximations from \cref{def:convex_approximations}. Then,
\begin{align*}
    &\| Q^\Prob - Q^{\hat{\Prob}} \|_\infty \leq \sum_{i=1}^m \bar{q}_i \E^{\Prob_{-i}}\Big[ H\big( \totvar f_i(\cdot|\xi_{-i}) \big) \Big] \ \ \text{and } \\
    &\| Q^\Prob - Q^{\tilde{\Prob}_\alpha} \|_\infty \leq \sum_{i=1}^m \bar{q}_i \E^{\Prob_{-i}}\Big[ H\big( \totvar f_i(\cdot|\xi_{-i}) \big) \Big],
\end{align*}
where $\totvar f_i(\cdot|s_{-i})$ denotes the total variation of the conditional pdf of $\xi_i$, given $\xi_{-i} = s_{-i}$ (with $s_{-i}$ denoting the vector $s$ without its $i^{th}$ element); $\bar{q}_i := q_i^+ + q_i^-$, $i=1,\ldots,m$; and $H : \R_+ \to \R$ is defined by
\begin{align*}
    H(t) = \begin{cases}
    t/8, &\text{if } 0 \leq t \leq 4, \\
    1 - 2/t, &\text{if } t > 4.
    \end{cases}
\end{align*}
\end{lemma}
\begin{proof}
    The error bounds follow from Theorems~6 and 5, respectively, in \cite{romeijnders2016tu}.
\end{proof}

\begin{example}\label{ex:compare_error_bounds}
    Consider the $\alpha$-approximation in the one-dimensional case ($m=1$) and suppose that $q^+ \geq q^-$. Then, \cref{lemma:totvar_bound}'s second inequality reduces to $\| Q^{\Prob} - Q^{\tilde{\Prob}_\alpha} \|_\infty \leq q^+ H(\totvar f)$, where $f$ is the pdf of $\xi$ under $\Prob$,  while \cref{cor:error_bounds_concrete} reduces to $\| Q^{\Prob} -Q^{\tilde{\Prob}_\alpha} \|_\infty \leq \textcolor{white}{\tilde{\tilde{\tilde{\tilde{\textcolor{black}{q^+}}}}}} \sqrt{2 W_1(\Prob, \tilde{\Prob}_\alpha)}$, using the fact that $W_1(\Prob, \tilde{\Prob}_\alpha) \leq 1/2$ by an analogue of \cref{lemma:Wass_dist_max_m/4}. 
    
    Both bounds have the same, trivial, worst-case value of $q^+$ by letting $\totvar f \to \infty$ and $W_1(\Prob, \tilde{\Prob}_\alpha) = 1/2$. In some cases, the Wasserstein bound performs better than the total variation bound. For instance, the Wasserstein bound correctly recognizes that $\| Q^{\Prob} - Q^{\tilde{\Prob}_\alpha} \|_\infty = 0$ if the conditional distribution of $\xi$ given $\xi \in [\alpha + k, \alpha + k + 1)$ is uniform for each $k \in \Z$ under $\Prob$. The Wasserstein bound reduces to zero, while the total variation bound does not; for instance, it equals $q^+/4$ if $\xi \sim U(\alpha, \alpha+1)$ under $\Prob$. There also exist cases in which the Wasserstein bound converges to zero while the total variation bound converges to its maximum value of $q^+$. For example, suppose that under $\Prob_n$, $\xi$ is uniformly distributed over $\cup_{k=0}^{2^{n}-1} (k 2^{-n}, (k+1)2^{-n})$, for $n \in \Z$. Then, indeed, $\lim_{n \to \infty} W_1(\Prob, \tilde{\Prob}_\alpha) = 0$, while $\lim_{n \to \infty} \totvar f_n = \infty$. However, for many ``nicely behaved'' continuous distributions $\Prob$, the total variation bound performs best. For instance, for a standard normal distribution the Wasserstein bound with $\alpha = 0$ equals $0.37 q^+$, while the total variation bound equals $0.10 q^+$.  \qedexample
\end{example}

Hence, we find that the performance of our error bounds can differ substantially from the performance of the error bounds from the literature. Nevertheless, both have a similar intuitive interpretation in terms of the dispersion of the distribution $\Prob$. The Wasserstein bound is small if $W_1(\Prob, \tilde{\Prob}_\alpha)$ is small, which is the case if the conditional distribution of $\xi$ given $\xi \in [\alpha + k, \alpha + k + 1)$, is close to a uniform distribution for every $k \in \Z$. Note that the Wasserstein bound equals zero if all these conditional distributions are indeed uniform, which is arguably the ``most dispersed'' case. Hence, similar to the total variation error bound, we can interpret the Wasserstein error bound by saying that the approximation is good if the distribution of $\xi$ is highly dispersed. Based on the examples above, the Wasserstein bound is better at recognizing dispersion \textit{within} unit intervals, while the total variation bound is better at recognizing dispersion \textit{between} unit intervals.

\section{Pragmatic moment-based DRSIR} \label{sec:pragmatic_moment_DRSIR}

Our pragmatic approach is not restricted to the Wasserstein setting. In principle, it can be applied to DRSIR models with an arbitrary uncertainty set $\mathcal{U}$. In this section we illustrate this by studying pragmatic DRSIR in a setting with an uncertainty set based on (generalized) moment conditions. We first outline our approach in primal form and show how it differs from standard moment-based DRSIR. Next, we derive the dual, which we use to sketch an algorithm based on row generation.

\subsection{Primal formulation} \label{subsec:moment_primal}

We consider the moment-based uncertainty set
\begin{align}
    \mathcal{U}_K :=  \left \{ \Prob \in \mathcal{P}(\R^m) \ \left | \ \int_{\R} g_i^k(s_i) \Prob_i(ds_i) = M_i^k, \quad k \in K_i, \ i=1,\ldots,m \right. \right \},  \label{eq:def_U_K}
\end{align}
with $M_i^k \in \R$ and $g_i^k : \R \to \R$, $k \in K_i$, with $K_i$ a finite index set, $i=1,\ldots,m$.
\begin{remark}
We only consider \textit{marginal} moment conditions, i.e., conditions on the moments of the marginal distributions $\Prob_i$, $i=1,\ldots,m$. In principle, our approach can be applied to joint moment conditions, too, but doing so adds complications without additional insight. 

\end{remark}

Similarly as for the Wasserstein-based uncertainty set, the corresponding DRSIR function $\mathcal{Q}_K(x) = \sup_{\Prob \in \mathcal{U}_K} \E^{\Prob}\big[ v(\xi,x) \big]$ is generally non-convex. For this reason, we define our pragmatic uncertainty set by restricting $\mathcal{U}_K$ to elements of $\mathcal{C}$. 
The pragmatic uncertainty set is 
\begin{align*}
    \tilde{\mathcal{U}}_K := \mathcal{U}_K \cap \mathcal{C},
\end{align*}
and the pragmatic DRSIR function is
\begin{align}
    \tilde{\mathcal{Q}}_K(x) := \sup_{\Prob \in \tilde{\mathcal{U}}_K} \E^{\Prob}\big[ v(\xi,x) \big], \qquad x \in \R^m. \label{eq:def_Q_tilde_K}
\end{align}
Writing the maximization problem in~\cref{eq:def_Q_tilde_K} as an infinite-dimensional LP, we can show that it is equivalent to a distributionally robust \textit{continuous} recourse problem under \textit{adjusted} moment conditions.

\begin{theorem} \label{thm:moment-based_DRSIR}
Consider the pragmatic DRSIR function $\tilde{\mathcal{Q}}_K$ from \cref{eq:def_Q_tilde_K}. Then,
    \begin{align*}
        \tilde{\mathcal{Q}}_K(x) := \sup_{\Prob \in \bar{\mathcal{U}}_K} \E^{\Prob}\big[ \hat{v}(\xi,x) \big], \qquad x \in \R^m,
    \end{align*}
    where $\hat{v}$ is the function from \cref{cor:QP=QhatPbar} and 
    \begin{align*}
        \bar{\mathcal{U}}_K := \left \{ \bar{\Prob} \in \mathcal{P}(\R^m) \ \left | \ \int_{\R} \hat{g}_i^k(s_i) \bar{\Prob}_i(ds_i) = M_i^k, \quad k \in K_i, \ i=1,\ldots,m \right. \right \},
    \end{align*}
    with $\hat{g}_i^k := \gamma \circ g_i^k$, $k \in K_i$, $i=1,\ldots,m$.
\end{theorem}
\begin{proof}
    By definition, $\Prob \in \mathcal{C} \iff \Prob_i = \gamma(\bar{\Prob}_i)$, $i=1,\ldots,m$, for some $\bar{\Prob} \in \mathcal{P}(\R^m)$. Let such a $\Prob \in \mathcal{C}$ with corresponding $\bar{\Prob}$ be given. Then, by \cref{cor:QP=QhatPbar} we can rewrite the conditions defining $\mathcal{U}_K$ as
    \begin{align*}
        \int_{\R} g^k_i(s_i) \gamma(\bar{\Prob}_i)(ds_i) = M^k_i \iff \int_{\R} \gamma(g^k_i)(s_i) \bar{\Prob}_i(ds_i) = M^k_i, \quad k \in K_i, \ i=1,\ldots,m.
    \end{align*}
    Moreover, by the same corollary and separability of SIR, we can write the objective as 
    \begin{align*}
        \E^{\Prob}\big[v(\xi,x)\big] &= \sum_{i=1}^m \E^{\Prob_i}\big[ v_i(\xi_i, x_i) \big] = \sum_{i=1}^m \E^{\gamma(\bar{\Prob}_i)}\big[ v_i(\xi_i, x_i) \big] \\
        &= \sum_{i=1}^m \E^{\bar{\Prob}_i}\big[ \hat{v}_i(\xi_i, x_i) \big] = \E^{\bar{\Prob}}\big[ \hat{v}(\xi,x) \big].
    \end{align*}
    Combining these facts and writing $\tilde{Q}_K(x)$ as an optimization problem over $\bar{\Prob}$ instead of $\Prob$, we obtain the result.
\end{proof}

We can use \cref{thm:moment-based_DRSIR} to compare our pragmatic DRSIR approach with standard moment-based DRSIR. We observe two changes. First, $v_i$ has been replaced by the \textit{convex} function $\hat{v}_i = \gamma(v_i)$, $i=1,\ldots,m$, corresponding to a \textit{continuous} recourse model. So, the non-convexity issues caused by integrality restrictions have again disappeared. Second, the functions $g^k_i$ defining the moment conditions have been replaced by the functions $\hat{g}^k_i = \gamma(g^k_i)$, $k \in K_i$, $i=1,\ldots,m$. Thus, our pragmatic DRSIR model is equivalent to a distributionally robust \textit{continuous} recourse problem with \textit{adjusted} moment conditions.

Observe that both $v_i$ and $g^k_i$ are transformed by the UIMA transformation $\gamma$. The interpretation is that our pragmatic approach is equivalent to ``smoothing'' these functions by means of the UIMA transformation $\gamma$. To understand what this means for the moment conditions, we provide an example.

\begin{example} \label{ex:U_K_tilde}
    Consider a one-dimensional setting ($m=1$). We define $\mathcal{U}_K$ to be the set of all distributions with a fixed mean $\mu$ and mean absolute deviation $d$. Suppressing index $i$, $\mathcal{U}_K$ is of the form \cref{eq:def_U_K}, with $K=\{1,2\}$, where we define $g^1(s) = s$, $s \in \R$, with $M^1 = \mu$; and $g^2(s) = |s - \mu|$, $s 
    \in \R$, with $M^2 = d$. Then, the adjusted moment conditions corresponding to the pragmatic DRSIR model involve functions given by
    \begin{align*}
        \hat{g}^1(s) &= \int_{s - 1/2}^{s + 1/2} g^1(t) dt = s, \quad s \in \R, \quad \text{and} \\
        \hat{g}^2(s) &= \int_{s - 1/2}^{s + 1/2} g^2(t) dt = \begin{cases}
            (s - \mu)^2 + 1/4, &\text{if } \mu - 1/2 \leq s \leq \mu + 1/2, \\
            |s - \mu| &\text{otherwise}.
        \end{cases}
    \end{align*}
    Note that in this example, $g^1$ does not change (it was already smooth), while $g^2$ is smoothed by the UIMA transformation $\gamma$.  \qedexample
\end{example}

\subsection{Dual formulation and algorithm sketch}

We conclude this section with a sketch of a row-generation algorithm for pragmatic moment-based DRSIR, with the goal of showing that our approach yields a tractable problem. For this purpose, we first derive a dual formulation for the pragmatic DRSIR function, $\tilde{\mathcal{Q}}_K(x)$.

Observing that $\bar{\Prob} \in \mathcal{P}(\R^m) \iff \bar{\Prob} \in \mathcal{M}_+(\R^m), \ \int_{\R}\bar{\Prob}_i(ds_i) = 1, \ i=1,\ldots,m$, we introduce dual variables $\pi_i$, $i=1,\ldots,m$, for the equality constraints above and $\nu^k_i$, $k\in K_i$, $i=1,\ldots,m$, for the constraints defining $\bar{U}_K$. Then, the dual formulation for $\tilde{\mathcal{Q}}_K(x)$ is given by
\begin{subequations}\label{eq:Q_K_tilde_dual}
\begin{align}
    \inf_{\nu_i^k, \pi_i} \Big\{ &\sum_{i=1}^m \Big( \sum_{k \in K_i} M_i^k \nu_i^k + \pi_i \Big) \label{eq:Q_K_tilde_dual_1} \\
    \text{s.t.} \ \ \ &\sum_{k\in K_i} \hat{g}_i^k(s_i) \nu_i^k + \pi_i \geq \hat{v}_i(s_i,x_i), \quad \forall s_i \in \R, \ i=1,\ldots,m \Big\}. \label{eq:Q_K_tilde_dual_2}
\end{align}
\end{subequations}
Strong duality holds if the primal-dual pair above satisfies regularity conditions (e.g., a constraint qualification) \cite{shapiro2009semi}. This will depend on the particular choice of functions $g_i^k$, $k \in K_i$, $i=1,\ldots,m$. Here we assume that such regularity conditions are satisfied and hence, strong duality holds. Then, we can combine the dual formulation above with the outer minimization over $x$, yielding the result below.

\begin{proposition}\label{prop:pragmatic_moment_DRSIR}
Consider the pragmatic moment-based DRSIR problem \linebreak $\min_{x \in X} \{ c^T x + \tilde{\mathcal{Q}}_K(x) \}$. Suppose that strong duality holds between the primal-dual formulations for $\tilde{\mathcal{Q}}_K(x)$ for all $x \in \R^m$. Then, the pragmatic DRSIR problem is equivalent to
\begin{subequations}\label{eq:pragmatic}
\begin{align}
        \inf_{x, \nu_i^k, \pi_i} \Big\{ &c^T x + \sum_{i=1}^m \Big( \sum_{k \in K_i} M_i^k \nu_i^k + \pi_i \Big) \label{eq:pragmatic_1} \\
    s.t. \ \ \ &\sum_{k\in K_i} \hat{g}_i^k(s_i) \nu_i^k + \pi_i \geq \hat{v}_i(s_i,x_i), \quad \forall s_i \in \R, \ i=1,\ldots,m \Big\}.\label{eq:pragmatic_2}
\end{align}
\end{subequations}
\end{proposition}
\begin{proof}
    The proof follows directly from the analysis above.
\end{proof}

Similar as for the pragmatic Wasserstein DRSIR problem from \cref{thm:Q_hat_dual}, the dual formulation in \cref{prop:pragmatic_moment_DRSIR} is naturally solved using a row generation algorithm. The master problem then is 
\cref{eq:pragmatic} with only a selection of its constraints included. Which constraints are included is determined by solving the subproblem 
\begin{align}
    \sup_{s_i \in \R} \Big\{ \hat{v}_i(s_i,x) - \sum_{k \in K_i} \hat{g}_i^k(s_i) \nu_i^k - \pi_i \Big\}, \label{eq:subproblem}
\end{align}
for every $i=1,\ldots,m$.

In this algorithm sketch, we immediately see two computational benefits of our pragmatic approach: one related to the master problem and one related to the subproblems. First, observing that $\hat{v}_i(s_i,\cdot)$ is a convex, piecewise linear function with three pieces, the constraints in \cref{eq:pragmatic_2} can be rewritten as \textit{linear} constraints. Note that this would not be the case if we had pursued the standard moment-based DRSIR problem, which uses the non-convex function $v_i$ instead. Second, in contrast with~$g^k_i$, the adjusted moment function $\hat{g}^k_i$ is smooth, i.e.,  differentiable with derivative
\begin{align*}
    \tfrac{d}{ds_i}\hat{g}_i^k(s_i) = g_i^k(s_i + 1/2) - g_i^k(s_i - 1/2), \quad s_i \in \R.
\end{align*}
In conjunction with piecewise linearity of $\hat{v}_i$ in $s_i$, this can be of use when solving subproblem~\cref{eq:subproblem}. So in sum, our pragmatic approach makes both the master problem and the subproblem easier to solve.

\section{Discussion} \label{sec:discussion}

In this paper we consider MIR models under distributional uncertainty. Standard DRMIR approaches inevitably inherit difficulties from both distributionally robust optimization and mixed-integer programming. We take a fundamentally different approach and argue that in the presence of integer variables, an alternative, \textit{pragmatic}, modeling approach is justified. The main goal of our approach is to obtain models that are easier to solve.

For the special case of SIR, we obtain such models by refining the uncertainty set ${\mathcal U}$. In particular,
we show how any standard DRSIR model can be transformed to a (convex) pragmatic DRSIR model by means of the UIMA transformation $\Gamma$ defining the set $\mathcal{C}$ of ``convexifying'' distributions. We explore this approach in detail for two classes of uncertainty sets in particular: a Wasserstein ball around a reference distribution and a generalized moment-based uncertainty set. We show how the resulting pragmatic DRSIR models relate to their standard counterparts and we sketch algorithmic approaches for solving them, based on dual formulations. 

An important result from our analysis is the fact that applying the UIMA transformation $\Gamma$ to the distribution $\Prob$ is equivalent to applying it to the value function $v$. This allows us to reformulate our pragmatic approach in terms of a ``convexified'' value function $\hat{v}$. This is particularly significant in the context of extending our pragmatic approach to more general MIR models. Currently, no analogues are known of the set of ``convexifying'' distributions $\mathcal{C}$ beyond the special case of SIR. However, analogues of $\hat{v}$ are readily available for general MIR models from the literature on convex approximations of MIR models. What is more, we showed in \cref{subsubsec:v_hat} that, even without embedding the model in a DRO framework, replacing the value function $v$ by the convex function $\hat{v}$ can hedge against ``fractional'' distributional uncertainty. In addition, it simultaneously results in more computationally attractive, convex models. We believe that this is a promising avenue for extending our ideas from the limited setting of SIR to much more general MIR models. Future research should be aimed at rigorously exploring the potential of this approach in realistic settings.

An important side result of our paper is the derivation of error bounds for convex approximations of SIR models, based on the Wasserstein distance between the original distribution and an approximating distribution. In contrast with error bounds from the literature, which only hold for continuous distributions, our error bounds hold for \textit{any} distribution. Interestingly, our error bounds have a similar interpretation as the error bounds from the literature: the convex approximation is better if the original distribution is more dispersed. The fact that we are able to extend this interpretation from continuous distributions to arbitrary distributions in the SIR case suggests that the same may be possible for more general MIR models. Hence, our results suggest that in general, convex approximations of MIR models may be better if the distribution is more dispersed, regardless of whether the distribution is continuous. Future research may be aimed at testing this claim rigorously by either deriving more general error bounds or by means of numerical experiments.

%
%
%

\appendix

\section{On \cref{remark:Gamma:details}}
\label{appendix:Gamma}

In this appendix we prove the claims in \cref{remark:Gamma:details} regarding the $m$-dimensional UIMA transformation $\Gamma$ from \cref{def:Gamma}. First we show that marginally, $\Gamma$ reduces to $\gamma$.
\begin{lemma}\label{lemma:Gamma_marginal}
Let $\Prob = \Gamma(\bar{\Prob})$ for $\bar{\Prob} \in \mathcal{P}(\R^m)$. Then, $\Prob_i = \gamma(\bar{\Prob}_i)$, $i=1,\ldots,m$.
\end{lemma}
\begin{proof}
    Let $F,\bar{F}$ be the joint cdfs corresponding to $\Prob, \bar{\Prob}$, respectively. Then, from the definition of $\Gamma$ it is clear that $F_i(s_i) = \lim_{s_j \to \infty, \ j \neq i} \Gamma \circ \bar{F}(s) = \int_{s_i - 1/2}^{s_i + 1/2} \bar{F}_i(t_i) dt_i$ for every $s_i \in \R$ and hence,  $\Prob_i = \gamma(\bar{\Prob}_i)$,  $i=1,\ldots,m$.
\end{proof}

As a result, $\Gamma$ reduces to $\gamma$ for $m=1$ and hence, it is a proper generalization of $\gamma$. However, for $m > 1$, the relation between $\Gamma$ and $\mathcal{C}(\R^m)$ is slightly different than the relation between $\gamma$ and $\mathcal{C}(\R)$. In particular, we have $\gamma(\mathcal{P}(\R)) = \mathcal{C}(\R)$ (by the definition of $\mathcal{C}(\R)$), whereas for $\Gamma$ we have the strict inclusion $\Gamma(\mathcal{P}(\R^m) \subset \mathcal{C}(\R^m)$, if $m>1$. We first show that this is indeed true, and then show that practically speaking, it does not matter.
\begin{lemma}\label{lemma:Gamma_inclusions}
    For $m > 1$ we have $\Gamma(\mathcal{P}(\R^m)) \subseteq \mathcal{C}(\R^m)$, but $\mathcal{C}(\R^m) \nsubseteq \Gamma(\mathcal{P}(\R^m))$.
\end{lemma}
\begin{proof}
    Let $F \in \Gamma(\bar{F})$ for some $\bar{F} \in \mathcal{F}(\R^m)$. Then, it is not hard to see that $F_i(s_i) = \lim_{s_j \to \infty, \ j \neq i} \Gamma(\bar{F})(s) = \int_{s_i - 1/2}^{s_i + 1/2} \bar{F}_i(t_i) dt_i$. Hence, $F \in \mathcal{C}$. For the second statement, let $\Prob \in \mathcal{P}(\R^2)$ such that $\xi$ is uniformly distributed on $(0,1/2)^2 \cup (1/2, 1)^2$ under $\Prob$. Then, marginally, $\xi_1, \xi_2 \sim U(0,1)$ and hence, $\Prob \in \mathcal{C}(\R^2)$. However, there clearly does not exist $\bar{\Prob} \in \mathcal{P}(\R^2)$ such that $\Prob = \Gamma(\bar{\Prob})$, since the marginal distributions of $\Prob$ are not independent. Hence, $\mathcal{C}(\R^2) \nsubseteq \Gamma(\mathcal{P}(\R^2))$. We can construct an analogous counterexample for any $m > 1$.
\end{proof}
The intuition behind this result is that whether $F$ is an element of $\mathcal{C}(\R^m)$ depends only on the \textit{marginal} distributions $F_i$, $i=1,\ldots,m$. However, there are numerous joint distributions $F$ with these marginals (corresponding to different copulas), not all of which can be written as $\Gamma(\bar{F})$ for some $\bar{F} \in \mathcal{F}(\R^m)$. Nevertheless, there always exists at least one joint distribution that can be written in this way, as we show in the following lemma.

\begin{lemma} \label{lemma:Gamma_F_tilde}
    Let $\Prob \in \mathcal{C}(\R^m)$ with joint cdf $F$ be given. Then, there exists $\tilde{\Prob} \in \Gamma(\mathcal{P}(\R^m))$ with joint cdf $\tilde{F}$ such that $\tilde{F}_i = F_i$, $i=1,\ldots,m$. In particular, this holds for $\tilde{F}(s) := \Pi_{i=1}^m F_i(s_i)$, $s \in \R^m$.
\end{lemma}
\begin{proof}
    Since $F \in \mathcal{C}(\R^m)$, there exists $\bar{F}_i$, such that $F_i = \gamma(\bar{F}_i)$, $i=1,\ldots,m$. Let $\bar{F}(s) := \Pi_{i=1}^m \bar{F}_i(s_i)$, $s \in \R^m$. Then, it is not hard to see that
    \begin{align*}
        \tilde{F}(s) = \int_{s_1 - 1/2}^{s_1 + 1/2} \cdots \int_{s_m - 1/2}^{s_m + 1/2} \bar{F}(t) dt_m \cdots dt_1, \quad s \in \R^m,
    \end{align*}
    so indeed, $\tilde{\Prob} \in \Gamma(\mathcal{P}(\R^m))$.
\end{proof}
This result shows that restricting ourselves to $\Gamma(\mathcal{P}(\R^m))$ instead of $\mathcal{C}(\R^m)$ goes without loss of generality in terms of the marginal distributions and, by separability of SIR, of the objective function.


\section*{Acknowledgements}
We wish to thank Wim Klein Haneveld for his very fruitful comments on earlier versions of this paper.

\bibliographystyle{siamplain}
\bibliography{references}
\end{document}